\documentclass{amsart}
\usepackage{amssymb}
\usepackage{enumitem}
\usepackage{xcolor}
\usepackage{pinlabel}
\usepackage{hyperref}
\usepackage{thm-restate}

\newtheorem {theorem}{Theorem}[section]
\newtheorem {lemma} [theorem] {Lemma}
\newtheorem {proposition} [theorem] {Proposition}
\newtheorem {corollary} [theorem] {Corollary}

\newtheorem* {claim*}{Claim}

\newtheorem* {subclaim*}{Subclaim}

\newtheorem {notation} [theorem] {Notation}

\newtheorem {observation} [theorem] {Observation}
\newtheorem {convention} [theorem] {Convention}

\newtheorem {assumption} [theorem] {Assumption}

\theoremstyle{definition}
\newtheorem{remark}[theorem]{Remark}
\newtheorem {definition} [theorem] {Definition}

\newcommand{\mc}{\mathcal}
\newcommand{\from}{\colon\thinspace}
\newcommand{\st}{\mid}
\newcommand{\Homeo}{\operatorname{Homeo}}
\newcommand{\fix}{\operatorname{fix}}
\newcommand{\Hom}{\operatorname{Hom}}

\newcommand{\dvis}{d_{\operatorname{vis}}}
\newcommand{\dgam}{d_{\Gamma}}
\newcommand{\dtrip}{d_{\Xi}}
\newcommand{\gprod}[3]{({#1}\,|\,{#2})_{#3}}
\newcommand{\bgam}{\partial G}
\newcommand{\diam}{\operatorname{diam}}

\newcommand{\dHaus}{d_{\operatorname{Haus}}}
\newcommand{\idG}{\mathbf{e}}
\newcommand{\identity}{\mathrm{id}}
\newcommand{\PSL}{\mathrm{PSL}}
\newcommand{\C}{\mathbb{C}}
\newcommand{\minsep}{\operatorname{minsep}}
\newcommand{\triples}{\Xi}
\newcommand{\minus}{-}
\newcommand{\bN}{\mathbb{N}}

\newcounter{notes}

\title{Stability for hyperbolic groups acting on boundary spheres}

\author[K. Mann]{Kathryn Mann}
\address{Department of Mathematics, Cornell University, Ithaca, NY 14853, USA
}
\email{k.mann@cornell.edu}

\author[J.F. Manning]{Jason Fox Manning}
\address{Department of Mathematics, Cornell University, Ithaca, NY 14853, USA
}
\email{jfmanning@cornell.edu}

\begin{document}

\begin{abstract}
A hyperbolic group $G$ acts by homeomorphisms on its Gromov boundary.  We show that if $\bgam$ is a topological $n$--sphere the action is {\em topologically stable} in the dynamical sense: any nearby action is semi-conjugate to the standard boundary action.  
\end{abstract}

\maketitle

\section{Introduction}

Calabi--Weil local rigidity \cite{Calabi, Weil} (an important precursor to Mostow rigidity) states that, for $n \geq 3$, the action of the fundamental group of a hyperbolic $n$-manifold by conformal maps on the boundary sphere $S^{n-1}$ is locally rigid: any nearby conformal action is conjugate in $SO^+(n, 1)$ to the original action.    
Inspired by this, we investigate rigidity for the actions on boundary spheres of the broader class of all Gromov hyperbolic groups with sphere boundary.  These boundaries do not typically admit a natural conformal or even a $C^1$ structure, so the relevant notion of local stability is that from topological dynamics.  

Recall that an action $\rho_0 \from G \to \Homeo(X)$ of a group $G$ on a topological space $X$ is a {\em topological factor} of an action $\rho \from G \to \Homeo(Y)$ if there is a surjective, continuous map $h\from Y \to X$ such that $h \circ \rho = \rho _0 \circ h$.  Such a map $h$ is called a {\em semi-conjugacy}.  An action of a group $G$ on a topological space $X$ is {\em topologically stable} or {\em $C^0$ stable} if it is a factor of any sufficiently close action in the compact-open topology on $\Hom(G, \Homeo(X))$.   We prove the following.  

\begin{restatable}[Topological stability]{theorem}{maintheorem}\label{thm:main}
Let $G$ be a hyperbolic group with sphere boundary. Then the action of $G$ on $\bgam$ is topologically stable.  More precisely, given any neighborhood $V$ of the identity in the space of continuous self-maps of $S^n$, there exists a neighborhood $U$ of the standard boundary action in $\Hom(G, \Homeo(S^n))$ such that any representation in $U$ has $\rho_0$ as a factor, with semi-conjugacy contained in $V$.  
\end{restatable} 

In parallel with Calabi--Weil rigidity, this says that these boundary actions exhibit the strongest possible form of local rigidity.  While there is overlap in the groups considered (fundamental groups of closed hyperbolic manifolds are Gromov hyperbolic), our result is neither a special case nor a generalization of the classical case. We consider a much broader space of deformations -- actions by homeomorphisms rather than conformal maps -- but semi-conjugacy is of course weaker than conformal conjugacy, which one cannot hope for when considering general continuous deformations (see in particular the examples of \cite[Section 4]{BowdenMann}).

\subsection*{History and related results} ``Stability from hyperbolicity" is an important and recurring theme in dynamical systems.  For instance, hyperbolic (Anosov) diffeomorphisms are topologically stable, thanks to the well known shadowing lemma.
However, in much of the existing literature, hyperbolicity is described using some smooth or at least $C^1$ structure, while the actions we consider are typically not differentiable, only having H\"older regularity. 

Regarding boundary actions of groups, Sullivan's 1985 {\em Structural stability implies hyperbolicity for Klein\-ian groups} \cite{Sullivan},  characterizes convex-cocompact subgroups of $\PSL(2,\C)$ as those subgroups whose action on their limit set is stable under $C^1$ perturbations.  Sullivan uses the fact that group elements expand neighborhoods of points to produce a coding of orbits that is insensitive to perturbation.   This technique was recently generalized by Kapovich--Kim--Lee \cite{KKL} to a much broader setting, including Lipschitz perturbations of many group actions on metric spaces which satisfy a generalized version of Sullivan's expansivity condition.  

Matsumoto \cite{Matsumoto} gives a more robust form of rigidity for the actions of fundamental groups of compact surfaces on their boundary at infinity.  In this case the boundary is a topological circle, and Matsumoto's work implies that any deformation of such a boundary action is semi-conjugate to the original action.   Motivated by this, Bowden and the first author studied the actions of the fundamental groups of compact Riemannian manifolds on their boundaries at infinity, showing these satisfy a form of local rigidity.  Again hyperbolicity played a role, this time in the form of the Anosov property of geodesic flow on such negatively curved manifolds.   
  
Theorem \ref{thm:main} generalizes aspects of both Sullivan's and Matsumoto's program.  While hyperbolic groups acting on their boundaries are among the examples studied by Kapovich--Kim--Lee, their methods only apply to pertubations which continue to have Sullivan's expansivity, for instance Lipschitz-close actions.   General $C^0$ perturbations need not be Lipschitz close, so Sullivan's coding no longer applies, and we need an entirely new method of proof.  Our strategy is more in the spirit of \cite{BowdenMann}, but uses large-scale geometry in place of the Riemannian manifold structure and Anosov geodesic flow.  

Our focus on spheres is motivated in part by the fact that these are the most homogeneous group boundaries.  At the other end of the spectrum, Kapovich and Kleiner \cite{KapovichKleiner00} constructed hyperbolic groups that are {\em boundary rigid} in the sense that any homeomorphism of the boundary comes from the action of an element of the group.  These groups trivially satisfy local rigidity since $\Homeo(\partial G) \cong G$ is discrete.
By contrast, homeomorphisms of the sphere are very easy to perturb, each having an infinite dimensional family of deformations.   
The reader may consult \cite{CPbook} or \cite{BKsurvey} for more background on the dynamics of hyperbolic groups acting on their boundaries.

\subsection*{Scope}
Bartels, L\"uck and Weinberger \cite[Ex.5.2]{BLW10} give, for all $k\geq2$, examples of torsion-free hyperbolic groups $G$ with $\partial G = S^{4k-1}$ that are not the fundamental group of any smooth, closed, aspherical manifold (note that such an example with $\partial G = S^2$ would give a counterexample to the Cannon Conjecture).  These examples show that, even in the torsion free case, Theorem \ref{thm:main} is a strict generalization of the work of \cite{BowdenMann} on Riemannian manifold fundamental groups.  Of course, groups with torsion provide numerous other examples, and the tools introduced within the large scale geometric framework of the proof should be of independent interest.

\subsection*{Outline} 
The broad strategy of the proof is to translate the data of a $G$--action on $S^n$ into a $G$--action on a sphere bundle over a particular space quasi-isometric to $G$, then show that nearby actions can be related by a $G$-equivariant map between their respective bundles that is close to the identity on large compact sets.   This lets us promote metrically stable notions in coarse negative curvature (such as the property of a subset being bounded Hausdorff distance from a geodesic) into stability for the group action.  

In Section \ref{sec:background}, we collect general results and preliminary lemmas on hyperbolic metric spaces.  In Section \ref{sec:fdef} we construct the bundles and equivariant map advertised above, in the broader context of (not necessarily hyperbolic) groups acting on manifolds, which is the natural setting for this technique.   

To apply this technique to the proof of Theorem \ref{thm:main}, we need to find a suitably nice space $X$ with a proper, free, cocompact action of $G$ by isometries.   If $G$ is torsion free, the space of distinct triples in $\bgam$ is a natural choice, but if $G$ has torsion the action of $G$ on triples may not be free.   We remedy this in sections \ref{sec:reductions} and \ref{sec:Xkappa}, first reducing to the case where $G$ acts faithfully on $\partial G$, and then showing that one may remove a small neighborhood of the space of fixed boundary triples without losing too much geometry, giving a suitable space to use in the rest of the proof.  We also prove several technical lemmas on the triple space for reference in later sections.  

Section \ref{sec:quasigeodesic} sets up the main proof and specifies a neighborhood of the boundary action where Theorem \ref{thm:main} holds.  The bundles from Section \ref{sec:fdef} come with natural topological foliations, 
and Section \ref{ss:geodesicintersection} shows that the image of one of these foliations in the source space (whose leaves are parameterized by points of $\bgam$) intersects leaves in the target along coarsely geodesic sets.    Section \ref{sec:endpoint} shows that the endpoints of these coarse geodesics depend only on the original leaf, thus giving a map $h$ from the leaf space of one foliation to the Gromov boundary of $X$.  Both the leaf space and the Gromov boundary of $X$ are canonically homeomorphic to $\bgam$, and the map $h$ is our semi-conjugacy.

\subsection*{Acknowledgements}
K.M. was partially supported by NSF grant DMS 1844516 and a Sloan fellowship.
J.M. was partially supported by Simons Collaboration Grant \#524176.

\section{Background} \label{sec:background}
We set notation, collect some general results on hyperbolic metric spaces and prove some preliminary lemmas needed for the main theorem.    

\subsection{Setup}
We fix the following notation.  $G$ denotes a non-elementary hyperbolic group; in this section we do not require $\bgam$ to be a sphere.  We fix a generating set $\mathcal{S}$, which gives us a Cayley graph $\Gamma$ and metric $\dgam$ on $\Gamma$.  Vertices of $\Gamma$ are identified with group elements.  In particular the identity $\idG$ is a vertex of $\Gamma$.
The metric $\dgam$ is $\nu$--hyperbolic (in the sense that geodesic triangles are $\nu$--thin) for some $\nu>0$.  We will fix a constant $\delta \ge \nu$ with some other convenient properties later.  The Gromov boundary of the group is denoted $\bgam$, this is of course equal to the Gromov boundary of $\Gamma$.

We write $\gprod{x}{y}{z}$ for the Gromov product of $x$ and $y$ at $z$.   
The point $z$ must lie in $\Gamma$, but $x,y$ may be in $\Gamma\cup\bgam$, using the standard definition of the Gromov product at infinity (see eg \cite[III.H.3.15]{BH99}), as follows: 
  \[ \gprod{x}{y}{p} = \sup\left\{\left. \liminf_{i,j\to \infty} \gprod{x_i}{y_j}{p}\ \right|\ \lim_{i\to\infty} x_i = x, \lim_{i\to\infty}y_i =y\right\}.\] 

We also fix a \emph{visual metric} $\dvis$ on $\bgam$.  This means a metric so that there are constants $\lambda>1$ and $k_2>k_1>0$ satisfying, for all $a,b\in \bgam$,
\begin{equation}
  \label{eq:bilip}
  k_1 \lambda^{-\gprod{a}{b}{\idG}}\le \dvis(a,b)\le k_2\lambda^{-\gprod{a}{b}{\idG}}.
\end{equation}
(See \cite[III.H.3]{BH99} or \cite[7.3]{GhH90} for more details, including the existence of such a metric.)
Unless otherwise specified, all metric notions  in $\bgam$  (such as balls $B_r(p)$) will be defined using this visual metric.
Occasionally, when specializing to $\bgam = S^n$ we also make use of the standard round metric on $S^n$.

We will need to use the following lemma for estimating the Gromov product of points at infinity in a $\nu$-hyperbolic space.   
\begin{lemma}\label{lem:gprodest}
  Let $M$ be $\nu$--hyperbolic, and let $p\in M$.  Let $\alpha,\beta\in \partial M$ be represented by geodesic rays $\gamma_\alpha,\gamma_\beta$ starting at $p$.  For any points $a\in \gamma_\alpha, b\in\gamma_\beta$,
  \[ \gprod{\alpha}{\beta}{p}\ge \gprod{a}{b}{p} - 2\nu.\]
\end{lemma}
\begin{proof}
From the definition of Gromov product at infinity, it follows that
  \[ \gprod{\alpha}{\beta}{p} \ge \liminf_{s,t\to\infty} \gprod{\gamma_\alpha(s)}{\gamma_\beta(t)}{p}.\]
To simplify notation, consider $p$ as a basepoint and write $|\cdot|$ for $d_M(p,\cdot)$.
  Suppose that $c$ is a point on $\gamma_\alpha$ and $d$ a point on $\gamma_\beta$ so that $|c|>|a|$ and $|d|>|b|$.
  We want to show that $\gprod{c}{d}{p} \ge \gprod{a}{b}{p} - 2\nu$.

  Consider the triple of Gromov products
   $\gprod{c}{d}{p}$,  $\gprod{b}{d}{p} = |b|$, and  $\gprod{c}{b}{p}\le |b|$.
  As is well known it follows from $\nu$--hyperbolicity that any one of such a triple is bounded below by $\nu$ less than the minimum of the other two, so we have
  \begin{equation}\label{eq:gprod1} \gprod{c}{d}{p} \ge \gprod{c}{b}{p} - \nu.\end{equation}
  Considering next the triple
  $\gprod{c}{b}{p}$, $\gprod{c}{a}{p} = |a|$, and  $\gprod{a}{b}{p} \le |a|$,
  we conclude
  \begin{equation}\label{eq:gprod2} \gprod{c}{b}{p} \ge \gprod{a}{b}{p} - \nu.\end{equation}
  Together \eqref{eq:gprod1} and \eqref{eq:gprod2} give $\gprod{c}{d}{p} \ge \gprod{a}{b}{p} - 2\nu$, as desired.
\end{proof}

\subsection{The space of triples} 
Write $\triples$ for the space of ordered distinct triples of points in the Gromov boundary $\partial \Gamma = \bgam$.  
We use the following well known property.  

\begin{proposition}[Convergence group property (see \cite{Bowditch98, Tukia98} and \cite{GromovEssay} 8.2.M)] \label{prop:convergence_group}
$G$ acts properly discontinuously and cocompactly on $\triples$, and each point $a \in \bgam$ is a conical limit point, meaning that there exists $\{g_i\}_{i\in\bN}\subset G$ and $p \neq q \in \bgam $ such that $g_i(a) \to p$ and $g_i(z) \to q$ for all $z \in \bgam - \{a\}$.   
\end{proposition}

The following definition can be thought of as giving a coarse projection map from $\triples$ to $\Gamma$.  (Compare \cite[8.2.K]{GromovEssay}.)  When $S$ is a subset of a metric space, we use the notation $N_r(S)$ to indicate the open $r$--neighborhood of $S$.

\begin{definition}[Coarse projection]\label{def:projection}
  For each $r>0$ we define a projection map $\pi_r$ from $\triples$ to subgraphs of $\Gamma$ as follows.  For $(a,b,c)\in \triples$ let $\mathcal{G}(a,b,c)$ be the set of geodesics in $\Gamma$ with endpoints in $\{a,b,c\}$. 
For each $r>0$ define $\pi_r(a,b,c) \subset \Gamma$ to be the smallest subgraph of $\Gamma$ containing $\bigcap_{\gamma  \in \mathcal{G}(a,b,c)} N_{r-1}(\gamma).$

If $Z$ is a subset of $\triples$, we define $\pi_r(Z) = \bigcup_{z \in Z} \pi_r(z)$.  For $s\in \Gamma$, we define $\pi_r^{-1}(s) = \{(a,b,c)\st s\in \pi(a,b,c)\}$; for $S\subset \Gamma$, define $\pi_r^{-1}(S) = \bigcup_{s\in S} \pi_r^{-1}(s)$.
\end{definition}

\begin{remark} \label{rem:projection}
  If $r$ is sufficiently large (depending on the hyperbolicity constant of $\Gamma$), then $\pi_r(a,b,c)$ is always nonempty.  
Moreover, for any $x\in \pi_r(a,b,c)$ and any geodesic $\gamma$ with endpoints in $\{a,b,c\}$, we have $\dgam(x,\gamma)\le r$.  We will make frequent use of this estimate. 
\end{remark}

\begin{lemma}
  \label{lem:diambound}
  For every $r\ge 0$, there is a $Q(r) \ge 0$ so $\diam(\pi_r(a,b,c))\le Q(r)$ for all $(a,b,c)\in \triples$.  
\end{lemma}

\begin{proof}
  Recall $\Gamma$ is $\nu$--hyperbolic, meaning that triangles are $\nu$--thin.
  Let $(a,b,c) \in X$ be given.  
Fix bi-infinite geodesics $[a,b], [b,c]$ and $[a,c]$ in $\Gamma$.  Approximate this ideal triangle by a triangle in $\Gamma$ by choosing points $a'$ $b'$ and $c' \in \Gamma$ on the geodesics $[a,b], [b,c]$ and $[a,c]$, respectively, satisfying $\dgam(a', [a, c]) < \nu$ and $\dgam(a', [b, c]) > r + 2 \nu$, and such that the same two inequalities also hold when the letters $a, b, c$ are cyclically permuted.  Then $\pi_r(a,b,c)$ is a subset of 
\[ S_r :=  N_{r + \nu}([a',b']) \cap N_{r + \nu}([b',c']) \cap N_{r + \nu}([c',a']) \]
so it suffices to show this set has diameter bounded by $Q(r)$, for some suitable function $Q$.  

Let $p$ denote the map from the triangle with sides $[a',b'], [b', c']$ and $[a', c']$ to a tripod witnessing that the triangle is $\nu$--thin, and let $Z$ be the preimage of the center, this is a set of three points with diameter at most $\nu$.  We now claim that $S_r$ lies in the $3r + 5 \nu$ neighborhood of $Z$, which is enough to prove the lemma.    To prove the claim, suppose $s \in S_r$, so there exist points $x_1, x_2$ and $x_3$ on $[a',b'], [b', c']$ and $[a', c']$ respectively with $\dgam(x_i, s) < r + \nu$.   Then for any $i = 1, 2, 3$, there exists some $j$ so that $p(x_i)$ and $p(x_j)$ lie on different prongs of the tripod, so there is a path between them passing through the midpoint $m$ of the tripod.  Thus, we have 
\[ \dgam(x_i, Z) = \dgam(p(x_i), m)  \leq \dgam(p(x_i), p(x_j)) \leq \dgam(x_i, x_j) + 2 \nu \leq 2r + 4 \nu \]
where the last inequality follows from the fact that $\dgam(x_i, s) < r + \nu$.    This proves the claim. 
\end{proof}

\begin{lemma}\label{lem:boundedK}
  Let $r\ge 0$.  For any compact $K\subset \triples$, the set $\pi_r(K)$ is bounded.
\end{lemma}
\begin{proof}
  Let $K\subset \triples$ be compact.  Increasing $r$ makes $\pi_r(K)$ larger.   Using Remark~\ref{rem:projection}, we can therefore assume that for every $(a,b,c)$ there is some $x\in \pi_r(a,b,c)$.

  For each $(a,b,c)\in K$, there is an open neighborhood $U$ of $(a,b,c)$ in $\triples$ such that, for each point $(u, v, w) \in U$, any geodesics joining points in $\{u,v,w\}$ come within $2 \nu + r$ of the point $x$.  In particular $x\in \pi_{r+2\nu}(u,v,w)$.  If $y\in \pi_r(u,v,w)$ then $\dgam(x,y)\le Q(r+2\nu)$, so $\pi_r(U)$ has diameter at most $2 Q(r+2\nu)$.

  By compactness we can cover $K$ with finitely many neighborhoods $U$ as in the last paragraph, so $\pi_r(K)$ is bounded.
\end{proof}

It will be convenient to choose a hyperbolicity constant for $\Gamma$ that simultaneously satisfies several properties.  The properties we use are collected in the following lemma.  

\begin{lemma}
  \label{lem:delta}
  There exists $\delta>0$ so that all of the following hold:
\begin{enumerate}[label = ($\delta$\arabic*)]
\item\label{item:deltathin} Every geodesic triangle in $\Gamma$ is $\delta$--thin.
\item\label{item:idealtriangles} Every geodesic bigon or triangle with vertices in $\Gamma\cup\bgam$ is $\delta$--slim.
\item\label{item:gromovproducts} For any point $p\in\Gamma$, and any $a,b,c\in\Gamma\cup\bgam$,
  \[ \gprod{a}{b}{p}\ge \min\{\gprod{a}{c}{p},\gprod{b}{c}{p}\}-\delta.\]
\item \label{item:surject} For all $p\in \Gamma$, $\pi_\delta^{-1}(p)$ is non-empty.
\item\label{item:ontoline} The set $\pi_\delta(\{a\}\times \{b\}\times (\bgam\minus\{a,b\}))$  contains every geodesic joining $a$ to $b$.
\end{enumerate}
\end{lemma}

\begin{proof}
  Since $\Gamma$ is $\nu$--hyperbolic, items~\ref{item:deltathin} and~\ref{item:idealtriangles} hold for any $\delta\ge 2\nu$.
   For item~\ref{item:gromovproducts} see \cite[III.H.3.17.(4)]{BH99}.   
Item~\ref{item:surject} follows from $G$-equivariance and the fact that $\pi_\delta(a,b,c)$ is nonempty when $\delta$ is large enough.    
For~\ref{item:ontoline}, suppose we are given a point $z$ on a geodesic $\gamma$ joining $a$ and $b$, take $c \in \bgam$ minimizing 
$\max\{\gprod{a}{c}{z}, \gprod{b}{c}{z} \}$.  Co-compactness of the action of $G$ allows one to bound this minimum from above, independently of $a, b$ and $c$, and this can be used to give an upper bound on the distance from $z$ to any geodesic joining $a$ or $b$ with $c$. 
\end{proof}

\begin{notation}
  For the rest of the paper we fix some $\delta>0$  so the conclusions of Lemma~\ref{lem:delta} hold,
 and denote the coarse projection $\pi_\delta$ by $\pi$.
  \end{notation}

\begin{definition}[Minimum separation]\label{def:minsep}
  For $x = (a,b,c)\in \triples$, we define
  \[\minsep(x) = \min\{\dvis(a,b),\dvis(a,c),\dvis(b,c)\}.\]  Notice that $1/\minsep$ is a proper function on $X$, so $\minsep$ is bounded away from zero on any compact set.  For a subset $D\subset \triples$, we define $\minsep(D) = \inf\{\minsep(x)\st x\in D\}$.
\end{definition}

\subsection{A criterion for a set to be close to a geodesic}
The following lemma gives a criterion for a piecewise geodesic curve to be close to a geodesic.  There are various similar statements in the literature (e.g. \cite[Lemma 4.2]{Minasyan05}, \cite[III.H.1.13]{BH99}), but this form will be convenient for us.
We use it to prove Lemma \ref{lem:bi-infinite} which is the main technical ingredient of this section.  
\begin{lemma}\label{lem:brokengeodesic}
 Let $X$ be a $\delta$--hyperbolic geodesic metric space, and let $l>0$.
 Suppose that $c$ is a piecewise geodesic in $X$ made of segments of length greater than $2l+8\delta$, with Gromov products in the corners at most $l$.  Let $\gamma$ be a geodesic with the same endpoints as $c$.  The Hausdorff distance between $\gamma$ and $c$ is at most $l+4\delta$.
\end{lemma}
  We remark that this lemma only uses $\delta$--hyperbolicity, and not the other properties from Lemma~\ref{lem:delta}.
\begin{proof}
  A standard argument, using only the fact that $\gamma$ is a geodesic in a $\delta$-hyperbolic space, shows that it is enough to prove $c$ is contained in the closed $(l + 3\delta)$--neighborhood of $\gamma$. We will not give the details as this is classical.  
    We write $c$ as a concatenation $c_1\cdots c_k$ of geodesics so each $c_i$ joins some $p_{i-1}$ to some $p_i$.  The endpoints of $\gamma$ are $p_0$ and $p_k$.
    If $k\le 2$, we are done by slimness of triangles, so we assume $k\ge 3$.

    Let $x$ be the farthest point from $\gamma$ on $c$, and let $M=d(x,\gamma)$.  Without loss of generality, we suppose that $M>2\delta$.
    It is then straightforward to show
    that $x$ is within $2\delta$ of some breakpoint $p_i$.  (Consider the triangle made up of the segment $c_i$ containing $x$, together with geodesics joining the endpoints of $c_i$ to a point $x'\in \gamma$ closest to $x$.)
    
  Since $M>2\delta$, the breakpoint $p_i$ cannot be either endpoint of the geodesic $\gamma$; in particular $i\notin\{0,k\}$.  There are two cases, depending on whether or not $i\in\{1,k-1\}$.

  We suppose $i\notin\{1,k-1\}$; the case $i\in \{1,k-1\}$ is similar but easier.
  By the assumption that segments are long, $d(x,\{p_{i\pm 1}\})>2l+6\delta$.
  Choose a geodesic $\sigma$ joining $p_{i-1}$ to $p_{i+1}$.  By the assumption on Gromov products in the corners, we have $\gprod{p_{i-1}}{p_{i+1}}{p_i}\le l$.   It follows that $d(x,\sigma)\leq l+\delta$.  Let $y$ be a closest point to $p_{i-1}$ in $\gamma$, and let $z$ be a closest point to $p_{i+1}$ in $\gamma$.  Choose geodesics $[y,z]\subset \gamma$, $[p_{i-1},y]$, and $[p_{i+1},z]$.  The point $x$ lies within $l+3\delta$ of some point $w$ on the union of these three geodesics.  We claim that $w\in [y,z]$, so we have $M\le l+3\delta$.

  Indeed, suppose that $w\in [p_{i-1},y]$ (the case $w\in [p_{i+1},z]$ being identical).  Now we have
\begin{align*}
  0 \leq d(x,y)-d(p_{i-1},y) & \leq d(x,w)+d(w,y) - (d(p_{i-1},w)+d(w,y))\\
& = d(x,w)-d(p_{i-1},w)\\
& \leq d(x,w)-(d(x,p_{i-1})-d(x,w))\\
& = 2d(x,w)-d(x,p_{i-1})\\
& \leq 2(l+3\delta) - d(x,p_{i-1}) < 0
\end{align*}
a contradiction.
We have thus established that $M \le l+ 3\delta$, and so $c$ lies in the $l+3\delta$--neighborhood of $\gamma$.
\end{proof}

\begin{definition}
  Let $r>0$, and let $M$ be a metric space.  A subset $S$ of $M$ is \emph{$r$--connected} if any two points $p,q$ of $S$ can be connected by a chain of points in $S$,
  \[ p = p_0,\ p_1,\ldots, p_k = q,\]
  so that $d_M(p_i,p_{i+1})\leq r$ for all $i$.
  An \emph{$r$--connected component of $S$} is a maximal subset of $S$ which is $r$--connected.
\end{definition}

\begin{lemma}
  \label{lem:bi-infinite}
  Let $H > 0$, and let $R > 24 H + 16 \delta$.
  Let $S\subset \Gamma$ be a $\frac{R}{4}$-connected set so that for every $s\in S$, there is a bi-infinite geodesic $\gamma_s$ satisfying: 
  \[ S\cap B_R(s)\subset N_H(\gamma_s),\mbox{ and } \gamma_s\cap B_R(s)\subset N_H(S).\]
  Then there is an oriented
  bi-infinite geodesic $\gamma$ so that
  
  \begin{enumerate}
  \item $\dHaus(\gamma,S)\le 3 H + 6\delta$; and
  \item for every $s\in S$, we may orient $\gamma_s$ so the Gromov products $\gprod{\gamma^{+\infty}}{\gamma_s^{+\infty}}{s}$ and $\gprod{\gamma^{-\infty}}{\gamma_s^{-\infty}}{s}$ are bounded below by $R - (4H+10\delta)$.  
  \end{enumerate} \end{lemma}
\begin{proof}
  Choose any $s_0$ in $S$, and let $\gamma_0 = \gamma_{s_0}$ be a bi-infinite geodesic as in the hypothesis, parameterized so that $\gamma_0(0)$ is within $H$ of $s_0$.  Since the points $\gamma_0(\pm \frac{R}{2})$ lie in $B_R(s)$, there are points $s_{\pm 1}\in S$ whose distances from $\gamma_0(\pm\frac{R}{2})$ are at most $H$.  Since $\dgam(s_0,s_{\pm 1})\le \frac{R}{2} + 2 H$ and $\dgam(s_{-1},s_1)\ge R- 2 H$, we deduce $\gprod{s_{-1}}{s_1}{s_0}\leq 3 H$.
In particular we have the following estimates:
  \begin{equation}\label{eq:estimateat0}
    \dgam(s_0,s_{\pm 1}) \ge \frac{R}{2}-2H\mbox{ and }\gprod{s_{-1}}{s_1}{s_0}\leq 3 H,
  \end{equation}
  Now we inductively find $s_i$ and $\gamma_i$ for all integers $i$.

  For clarity we focus on $i>1$.  The construction for $i<0$ is entirely analogous.
  Suppose we have chosen points $s_{-1},s_0,\ldots s_{i-1}$ in $S$, and that for each positive $j\le i-1$ we have chosen a 
  bi-infinite geodesic $\gamma_j = \gamma_{s_j}$, and some $t_j$ within $4H$ of $-\frac{R}{2}$ so that
  \begin{align*}\label{eq:inducthyp}
    \dgam(\gamma_{i-1}(0),s_{i-1})& \le H,\mbox{ and }\\
    \dgam(\gamma_{i-1}(t_{i-1}),s_{i-2})& \le H.
  \end{align*}
  Since $R>8H$, the number $t_{i-1}$ is negative.
  The point $\gamma_{i-1}(\frac{R}{2})$ lies in the $R$--ball around $s_{i-1}$, so we may choose a point $s_i\in S$ so that $\dgam(s_i,\gamma_{i-1}(\frac{R}{2}))\le H$.  Let $\gamma_i$ be the geodesic $\gamma_{s_i}$ provided by the hypothesis of the lemma.  We can assume that $\gamma_i(0)$ is within $H$ of $s_i$.  The distance $\dgam(s_{i-1},s_i)$ differs from $\frac{R}{2}$ by at most $2H$.  Thus for some $t_i$ of absolute value in $[\frac{R}{2}-4H,\frac{R}{2}+4H]$, we have $\dgam(\gamma_i(t_i),s_{i-1})\le H$.  We parameterize $\gamma_i$ so that $t_i<0$.  This completes the inductive construction.

  From the construction, we have
  \begin{equation}\label{eq:segmentbound}
    \dgam(s_i,s_{i+1})\le \frac{R}{2} + 2 H
  \end{equation}
  and
  \begin{equation*}
    \dgam(s_{i-1},s_{i+1}) \ge \frac{R}{2}+ |t_i| - 2 H \ge R - 6 H.
  \end{equation*}
  (The lower bound when $i=0$ is slightly better.)
  This implies a bound on Gromov products
  \begin{equation}\label{eq:gprodbound}
    \gprod{s_{i-1}}{s_{i+1}}{s_i} \le 5 H.
  \end{equation}

  Let $c_k$ be a piecewise geodesic formed by concatenating geodesics
  \[ [s_{-k},s_{-k+1}]\cdots [s_{k-1},s_k]. \]
  We verify the hypotheses of Lemma~\ref{lem:brokengeodesic} with $l = 5H$.  The inequality \eqref{eq:gprodbound} gives the bound on Gromov products in in the corners.
  The inequality \eqref{eq:segmentbound} gives that the segments $[s_i,s_{i+1}]$ have length at least $\frac{R}{2} - 2H > 10 H + 8 \delta = 2 l + 8\delta$ as required.  Thus if $\beta_k$ is the geodesic joining the endpoints of $c_k$, we have $\dHaus(c_k,\beta_k)\le H + 4\delta$.

  Since $\Gamma$ is proper, and the geodesics $\beta_k$ all pass through the $(H+4\delta)$--ball about $s_0$, they subconverge to a bi-infinite geodesic $\gamma$.  Notice that all the segments $[s_i,s_{i+1}]$ lie in the $(H+4\delta)$--neighborhood of $\gamma$.
  We will show this $\gamma$ satisfies the conclusions of the lemma.

  If $s\in S$, then there is a $\frac{R}{4}$--coarse path joining $s_0$ to $s$, that is to say there exist points $s_0=p_0 ,\  p_1,\  \ldots ,\ p_k = s$ in $S$ satisfying. 
  \[ \dgam(p_i,p_{i+1})\le \frac{R}{4},\ \forall i. \]

  We first claim that for each $p_i$, there is some $s_{j(i)}$ with $\dgam(s_{j(i)},p_i)\le \frac{R}{4}+2 H$.  Clearly this is true for $p_0$.  Arguing by induction, we see that $\dgam(p_{i+1},s_{j(i)})$ is at most $\frac{R}{2}+2 H$.  In particular there is a point $q$ on $\gamma_{s_{j(i)}}$ within $H$ of $p_{i+1}$.  Recalling that $\gamma_{s_{j(i)}}(0)$ lies within $H$ of $s_{j(i)}$, we see that $q = \gamma_{s_{j(i)}}(t)$ for some $t$ with
  \[|t| \le \frac{R}{2} + 3H<\frac{3}{4}R.\]
  Thus for some $t' \in \{-\frac{R}{2}, 0, \frac{R}{2}\}$, we have $|t-t'|\le \frac{R}{4}$.
    Since the points $\gamma_{s_{j(i)}}(\pm \frac{R}{2})$ are within $H$ of $s_{j(i)\pm 1}$, there is some $s_{j(i+1)}\in \{s_{j(i)-1},s_{j(i)},s_{j(i)+1}\}$ so that $\dgam(p_{i+1},s_{j(i+1)}) \le 2H + \frac{R}{4}$, as desired.
    
  Now let $s_j = s_{j(k)}$, so we have $\dgam(s,s_j)\le \frac{R}{4} + 2H$, and let $q$ be a closest point to $s$ on $\gamma_j$.  Without loss of generality we suppose that $q = \gamma_j(t)$ for $t\ge 0$.  Any quadrilateral with corners $s_j, s_{j+1},\gamma_j(0),\gamma_j(\frac{R}{2})$ is $2\delta$--thin, so there is a point $r$ on $[s_j,s_{j+1}]$ within $H+2\delta$ of $q$.  This point $r$ is within $H+4\delta$ of some point $z$ on $\gamma$.  Adding up the constants we have
  \begin{align*}
    \dgam(s,z) & \le \dgam(s,q) + \dgam(q,r) + \dgam(r,z)\\
    & \le H + H + 2\delta + H + 4\delta = 3H + 6\delta.
  \end{align*}
  This shows
  \begin{equation}\label{eq:H1}
    S \subset N_{3H + 6\delta}(\gamma).    
  \end{equation}

  Conversely, let $x\in \gamma$.  Then $x\in \beta_k$ for some $k$, and so for some $i$, there is a point $y\in [s_i,s_{i+1}]$ with $\dgam(x,y)\le H + 4\delta$.
  This point is within $H+2\delta$ of a point on $\gamma_i$, which is within $H$ of a point of $S$, so we have
  \begin{equation}
    \label{eq:H2}
    \gamma \subset N_{3H+6\delta} (S) .    
  \end{equation}

  Together, \eqref{eq:H1} and \eqref{eq:H2} imply the first statement of the Lemma, that is to say the bound on Hausdorff distance.
  It remains to show the statement about Gromov products.
  Breaking symmetry, we consider just the ray $\gamma_s|[0,\infty)$.  Let $y'$ be a point on $\gamma_s|[0,\infty)$ at distance $R$ from $s$.  Let $s'\in S$ be a point within $H$ of $y'$, and let $z'$ be a point on $\gamma$ within $3H + 6\delta$ of $s'$.  Let $\alpha$ be a ray starting at $s$ with limit point $\gamma_s^{+\infty}$, and let $\beta$ be a ray starting at $s$ with limit point $\gamma^{+\infty}$.  There are points $y$ on $\alpha$ and $z$ on $\beta$ which are within $\delta$ of $y'$, $z'$, respectively.  We have $\dgam(s,y)\ge R-\delta$, $\dgam(s,z) \ge R - (4H + 7\delta)$ and $\dgam(y,z) \le 3H + 8\delta$, so
  \[ \gprod{y}{z}{s} \ge \frac{1}{2}(R - \delta + R-(4H + 7\delta) - (4 H + 8\delta)) = R-(4H + 8\delta). \]
  Lemma~\ref{lem:gprodest} allows us to conclude $\gprod{\gamma_s^{+\infty}}{\gamma^{+\infty}}{s} \ge R - (4H + 10\delta)$ as desired.
\end{proof}

\section{An equivariant map from $X\times\bgam$ to itself.}\label{sec:fdef}
The first step in the proof of Theorem~\ref{thm:main} is the following construction, which can be thought of as a generalization of that in \cite[Lemma 3.1]{BowdenMann}. 
If $X$ is a space with a proper, free and cocompact action of $G$, and $\rho\from  G \to \Homeo(Y)$ an action of $G$ on a topological space, one can capture the information of this action as the holonomy of a foliated $Y$-bundle over $X/G$; this is simply the quotient of $X \times Y$ by the diagonal action of $G$.  Here the case of interest to us is when $Y = \bgam = S^n$.  The following proposition gives a construction of a ``nice" map between the foliated bundles associated to the boundary action $\rho_0$ and a small perturbation $\rho$.

The same proof works with any manifold $Y$ in place of $S^n$, and any two nearby actions of an arbitrary group $G$ on the space, so we state it in this general context, as follows. 

Let $Y$ be a metric space such that $\Homeo(Y)$ is metrizable and locally contractible.   For instance, one may take $Y$ to be any compact manifold, in which case local contractibility of $\Homeo(Y)$ follows from Edwards--Kirby \cite{EdwardsKirby71}.  
Metrizability of $\Homeo(Y)$ has the following easy consequence.  
\begin{observation}\label{obs:VfromW} 
  Let $W$ be a neighborhood of the identity in $\Homeo(Y)$, and let $F\subset \Homeo(Y)$ be finite.  Then there is a neighborhood $V \subset W$ of the identity so that the union
  \[ \bigcup_{f\in F} f V f^{-1} V \]
  lies in $W$.  
\end{observation}

Using this, we prove the following.  

\begin{proposition} \label{prop:good_map}
  Let $G$ be a group, $Y$ a metric space as above, and fix an action $\rho_0\from G \to \Homeo(Y)$.  Let $X$ be a metric space on which $G$ acts properly, freely and cocompactly by isometries.

  For any compact $K \subset X$  and $\epsilon>0$, there is a neighborhood $U$ of $\rho_0$ so that for each $\rho\in U$ there is a homeomorphism  $f^\rho\from X\times Y \to X\times Y$ with the following properties:
\begin{enumerate}
\item (Covers $\identity_X$) If $\pi_X$ is the projection from $X\times Y$ to $X$, then $\pi_X\circ f^\rho = \pi_X$.  In other words, $f^\rho$ covers the identity on $X$. 
\item (Equivariance)  For every $g \in G$ we have
  \[f^\rho(g\cdot x,\rho(g)\cdot\theta) = (g,\rho_0(g))\cdot f^\rho(x,\theta)  .\]
\item (Near flatness.) For any $\theta\in Y$ we have
  \[ f^\rho(K\times\{\theta\}) \subset X \times B_\epsilon(\theta). \] 
\end{enumerate}
\end{proposition}

\begin{proof}[Proof of Proposition~\ref{prop:good_map}]
  Since the action is proper, free, and cocompact, there is some $r>0$ so that every nontrivial element of $G$ moves every point of $X$ a distance at least $r$.  
  Choose a $G$--equivariant locally finite cover $\mc{U}=\{U_i\st i\in I\}$ of $X$ by open balls of radius $r/3$, and let $N$ be the nerve of $\mc{U}$.  Since $\mc{U}$ was $G$--equivariant and locally finite, the group $G$ acts cocompactly on the simplicial complex $N$.  Since any $g\in G\minus\{1\}$ moves every set $U_i$ off of itself, $G$ acts freely on $N$.  

  We choose a $G$--equivariant partition of unity $\{\phi_i\from X\to [0,1]\st i\in I\}$ subordinate to the cover $\mc{U}$.  This partition determines a proper $G$--equivariant map $\psi\from X\to N$.
  
  To define $f^\rho$, we first define a map $ \varphi^\rho \from N\to \Homeo(Y)$,  and then define
   \begin{equation}
    \label{eq:flatmapdef}
    f^\rho( x, \theta ) = \left(x, (\varphi^\rho \circ \psi (x) )(\theta)\right).
      \end{equation}
The map $\varphi^\rho$ will be $G$--equivariant with respect to the ``mixed'' left action of $G$ by homeomorphisms of $\Homeo(Y)$ given by 
  \begin{equation}\label{eq:MixedAction}
    g \cdot h = \rho_0(g) h \rho(g^{-1}).
  \end{equation}

\medskip  
 \paragraph{\textbf{Definition of $\varphi^\rho$.}}
 The definition of $\varphi^\rho$ is designed to keep track of the compact set $K$ and constant $\epsilon>0$ for the near flatness condition in the Proposition.  
 Let $D$ be a connected union of open simplices in $N$ which meets every $G$--orbit exactly once.  
 Let $K$ be a compact subcomplex of $N$, which we assume contains the closed star of any cell of $\overline{D}$.  (Note that any compact $C\subset X$ has $\psi(C)\subset K$ for some such complex.)  Let $\epsilon>0$.  
Let $S$ be the (finite) set of group elements $s$ so that $sD$ meets the closed star of some vertex in $\overline{D}$.  Let $F$ be the (still finite) set of group elements $g$ so that $gD\cap K$ is non-empty.

Letting $W = N_\epsilon(\identity)$, 
we choose a neighborhood $V$ as in Observation~\ref{obs:VfromW} so that $\rho_0(g)V\rho_0(g)^{-1}V$ lies in $W$ for all $g\in F$.  Now let $m = \dim(N)$.  Apply Observation~\ref{obs:VfromW} and local contractibility of $\Homeo(Y)$ to choose a nested sequence of contractible neighborhoods of $1$ inside $V$:
  \[ V_0\subset V_1 \subset \cdots \subset V_m \subset V \]
  so that for each $i$ and all $s\in S$, we have
  \begin{equation}\label{eq:V_i}
    \rho_0(s) V_i \rho_0(s)^{-1} V_i \subset V_{i+1}.
  \end{equation}
 By taking $\rho$ sufficiently close to $\rho_0$, we may assume that $\rho_0(g)\rho(g^{-1})$ lies in $V_0$ for all $g\in F$.  
We define $\varphi^\rho$ inductively over the $k$--skeleta of $N$ in such a way that $\varphi^\rho(\sigma) \subset V_k$ for every $k$--cell in the closed star of a vertex of $\overline{D}$. 

\medskip
\noindent \textit{0-skeleton.}   Define $\varphi^\rho$ on $N^{(0)}$ as follows.  If $v = g v_0$ for some $v_0\in D$, then
  \[ \varphi^\rho(v) = \rho_0(g)\rho(g^{-1}) .\]
  If $v$ lies in the closed star of some cell of $D$, then $g\in S$, so $\varphi^\rho(v)\in V_0$ as desired.

  \medskip
  \noindent \textit{Inductive step.}   Suppose that $\varphi^\rho$ has been defined on all $(k-1)$--cells, and let $\sigma$ be a $k$--cell.  We may write $\sigma = g\sigma_0$, where $\sigma_0$ is an open $k$--cell in $D$.  The map $\varphi^\rho$ has already been defined on the boundary of $\sigma_0$, and sends this boundary into $V_{k-1}$ by induction.  Using the contractibility of $V_{k-1}$ we extend $\varphi^\rho$ over $\sigma_0$ in such a way that $\varphi^\rho(\sigma_0)\subset V_{k-1}$.  We define $\varphi^\rho|_{\sigma (x)} = \rho_0(g)\varphi^\rho(g^{-1}(x))\rho(g^{-1})$.  Since the action on $N$ is free, there is no ambiguity in this definition.
  
  Now suppose that $\sigma$ lies in the closed star of some vertex of $\overline{D}$, so $\sigma = s\sigma_0$ for some  $s\in S$.  The set $\varphi^\rho(\sigma)$ lies in 
  \[ \rho_0(s) V_{k-1} \rho(s^{-1}) = \rho_0(s) V_{k-1} \rho_0(s)^{-1}\rho_0(s) \rho(s)^{-1}\subset \rho_0(s) V_{k-1} \rho_0(s)^{-1}V_0,\]
which lies in $V_k$ by \eqref{eq:V_i}.  Having verified the inductive hypothesis, we see that we can continue until we have defined $\varphi^\rho$ equivariantly on all of $N$.  Moreover, we have defined it so that $\varphi^\rho(\sigma_0)$ lies in the neighborhood $V$ for any $\sigma_0$ meeting $D$.

\medskip
 \paragraph{\textbf{Properties of $f^\rho$}}
  
Having defined $\varphi^\rho$, we define $f^{\rho}$ as in~\eqref{eq:flatmapdef}.
  \begin{equation*}
   f^\rho( x, \theta ) = \left(x, (\varphi^\rho \circ \psi (x) )(\theta)\right).
  \end{equation*}
 By definition, this covers the identity map on $X$.  
 To simplify notation, let $\Phi^\rho$ denote $\varphi^\rho \circ \psi$.   Note that $\Phi^\rho$ satisfies equivariance as $\varphi^\rho$ does.  
  This also gives equivariance of $f^\rho$, as follows: 
  \begin{align*}
    f^\rho(g\cdot x,\rho(g)\cdot\theta) & = \left( g\cdot x, \left(\Phi^\rho(g\cdot x)\rho(g)\right)\cdot \theta\right) \\
                                        & = \left( g\cdot x, \left(\rho_0(g)\Phi^\rho(x)\rho(g^{-1})\rho(g)\right)\cdot \theta\right) \\
                                        & = \left( g\cdot x, \left(\rho_0(g)\Phi^\rho(x)\right)\cdot \theta\right) \\
                                        & = \left( g,\rho_0(g)\right)\cdot \left(x,\Phi^\rho(x)\cdot \theta\right) \\
    & = \left( g,\rho_0(g)\right)\cdot f^\rho( x, \theta ).
  \end{align*}

It remains to check near flatness.   For any cell $\sigma$ of the larger compact complex $K$ there is some $g\in F$ and some $\sigma_0\subset D$ so that $\sigma = g \sigma_0$. 
Equivariance tells us that
  \begin{align*}
    \varphi^\rho(\sigma) & = \rho_0(g)\varphi^\rho(\sigma_0)\rho(g^{-1})\\
                  & = \rho_0(g)\varphi^\rho(\sigma_0)\rho_0(g)^{-1} \cdot \rho_0(g)\rho(g^{-1}) \\
    & \subset \rho_0(g) V \rho_0(g)^{-1} \cdot V \subset W.
  \end{align*}
  Since every homeomorphism in $W$ moves every point of $Y$ a distance of at most $\epsilon$, we have $f^\rho(\sigma\times\{\theta\})\subset B_\epsilon(\theta)$ as desired.
\end{proof}

\section{Reduction to the main case} \label{sec:reductions}
In this section we reduce to the case $\bgam = S^n$ for $n \geq 2$, and also to the case where $G$ 
acts faithfully on its boundary.   
The first reduction (to $n\ge2$) comes from combining work of Matsumoto with the Convergence Group Theorem.

\begin{proposition} 
Let $G$ be a hyperbolic group with circle boundary. Then the action of $G$ on $\bgam$ is topologically stable.  
\end{proposition} 
This can quite likely be derived from Matsumoto's original proof, as the main techniques are Euler characteristic and lifting to covers.
For completeness, we give a short argument using standard tools from circle dynamics.  
 
\begin{proof} 
Let $G$ be a hyperbolic group with circle boundary.  By the Convergence Group Theorem~\cite{Gabai,CassonJungreis94}, there exists a 
normal, finite index torsion-free subgroup $G'$ of $G$ that is isomorphic to the fundamental group of a closed surface.  Let $\rho$ be a perturbation of the standard boundary action $\rho_0$ of $G$.  
By \cite{Matsumoto}, there exists a continuous, surjective, degree one map $h: S^1 \to S^1$ such that $h \rho(g) = \rho_0(g) h$ for each $g \in G'$. 
if the action of $G'$ is minimal, then $h$ is a conjugacy.  If the action of $G'$ is not minimal, then there exists a unique invariant \emph{exceptional minimal set} $X$, homeomorphic to a Cantor set, and $h$ collapses the closure of each complementary interval to a point and is otherwise injective (See e.g. \cite[Proposition 5.6]{Ghys01}.)  It is also easy to see that $h$ varies continuously with $\rho$, so can be taken as close to the identity as desired by taking $\rho$ close to $\rho_0$.    

Since $G'$ is normal in $G$, the set $X$ is $\rho(G)$-invariant.  It follows that $G$ permutes the point-preimages of $h$.  From this we will now deduce that $h$ in fact defines a semiconjugachy intertwining the actions of $\rho_0(G)$ and $\rho(G)$.  To see this, we use the fact that attracting fixed points of elements of $\rho_0(G')$ are dense in $S^1$.  If $x$ is the attracting fixed point of $\rho_0(\gamma)$ for some $\gamma \in G'$, then $\rho_0(g)x$ is the attracting fixed point of $\rho_0(g\gamma g^{-1})$, an element which also lies in $G'$.   Thus, for any $y \in h^{-1}(x)$, we have $\rho(g)(y) \in h^{-1}\rho_0(g)(x)$; equivalently, $h \rho(g)(y) = \rho_0(g) h(y)$.  Since $h$ is continuous, and the union of preimages of attracting fixed points is dense in $S^1$, this shows that $h \rho(g) = \rho_0(g) h$ holds globally. 
\end{proof}

\begin{proposition}
  Suppose $G$ is a hyperbolic group with sphere boundary, and let $F<G$ be the subgroup of elements which act trivially on $\bgam$.  Then $G$ is topologically stable if and only if $G/F$ is topologically stable.
\end{proposition}
\begin{proof}
  Since $F$ is a finite normal subgroup, the canonical action of $G$ on its boundary factors through the canonical action of $G/F$ on its boundary, and these boundaries are the same.
  Let $N$ be the maximum order of an element of $F$.  
  By a theorem of Newman \cite{Newman31}, there is a neighborhood $U$ of the identity in $\Homeo(S^n)$
  so that any torsion element in $U$ has order greater than $N$.  Thus for any sufficiently small perturbation of the canonical action, the elements of $F$ will still act trivially, and so small perturbations of the canonical action of $G$ on its boundary are in one-to-one correspondence with small perturbations of $G/F$ on its boundary.  
\end{proof}

We therefore make the following assumptions for the remainder of the paper.
\begin{assumption}
  \label{assump:faithful} The hyperbolic group $G$ acts faithfully on its boundary $\bgam$.
\end{assumption}
\begin{assumption}
  \label{assump:notcircle} The boundary of $G$ is a topological sphere of dimension at least two.
\end{assumption}

\section{A space with a proper, cocompact and free action of $G$}\label{sec:Xkappa}
Recall that $\triples$ is the space of distinct triples in $\bgam$.  Since $\bgam$ is assumed to be a sphere of dimension $n\ge 2$ (Assumption~\ref{assump:notcircle}), the triple space $\triples$ is a connected $3n$--manifold.  The following is an easy consequence of \cite[Theorem 1.1]{AMN11}.
\begin{proposition}
  There is a proper $G$--invariant metric on $\triples$.
\end{proposition}

We fix such a metric $\dtrip$ now.  
If additionally no nontrivial element of $G$ fixes more than two points of $\bgam$, then the action of $G$ on $\triples$ is free and $\triples$ can play the role of the metric space $X$ in Proposition~\ref{prop:good_map}.  In general $G$ may not act freely on $\Xi$, so a different space is needed.  In this section, we show how to build such a space with a free action by deleting a small, $G$-equivariant neighborhood of the set of points in $\Xi$ with nontrivial stabilizer.

\begin{notation} 
For $g \in G$, we denote by $\fix(g) \subset \bgam$ the set of points fixed by the natural action of $g$.  We set
\[F = \left\{ (a, b, c) \in \triples \st \{a,b,c\} \subset \fix(g)  \text{ for some nontrivial } g \in G  \right\}.\]  
\end{notation}  
Recall a subset $S$ of a metric space $M$ is \emph{$\epsilon$--dense} if every point of $M$ is distance at most $\epsilon$ of a point of $S$.  Also recall that $N_\kappa(S)$ denotes the open $\kappa$--neighborhood of a set $S$.  If $S$ is a subset of $\triples$, this neighborhood is to be taken with respect to $\dtrip$.
Our first goal is to establish the following. 

\begin{proposition}[F is sparse] \label{prop:F_sparse}
  For any $\epsilon >0$, there exists $\kappa> 0$ such $\triples \minus N_{\kappa}(F)$ is $\epsilon$--dense in $\triples$.
\end{proposition} 
Since we will frequently refer to $\triples \minus N_{\kappa}(F)$, we fix the following notation.
\begin{notation}
  For $\kappa>0$ we let $X_\kappa = \triples \minus N_{\kappa}(F)$.
\end{notation}

Proposition \ref{prop:F_sparse} has the following useful corollary.
\begin{corollary}\label{cor:deltaOK}
  For sufficiently small $\kappa$, property~\ref{item:surject} of Lemma~\ref{lem:delta} still holds with $\pi_\delta$ replaced by the restriction of $\pi_\delta$ to $X_\kappa$.
\end{corollary}
\begin{proof}
  By equivariance and the fact that $\pi(x)$ is non-empty for every triple in $\triples$, Property~\ref{item:surject} holds as soon as $X_\kappa$ is non-empty.
\end{proof}

The proof of Proposition~\ref{prop:F_sparse} requires several preliminary results, starting with the following.   

\begin{lemma}\label{lem:MO}
For any torsion element $g \in G$, such that $\fix(g)$ has at least three points, there exists a quasi-convex subgroup $Q_g \subset G$ such that $\Lambda(Q_g) = \fix(g)$  
\end{lemma} 
Here we use the standard notation $\Lambda(Q_g)$ for the limit set of $Q_g$ in $\bgam$.  

\begin{proof}
This proof is adapted from an argument of Misha Kapovich \cite{Mathoverflow}.    
Fix a torsion element $g$.  Let $H$ be a maximal torsion subgroup pointwise fixing $\fix(g)$, and let $Q$ be the normalizer of $H$ in $G$.  We claim $Q$ is the stabilizer of $\fix(H) = \fix(g)$.  That $Q$ preserves this set is immediate.  For the reverse inclusion, if $f \in G$ preserves $\fix(H)$ then $fHf^{-1}$ pointwise fixes $\fix(H)$ as well.  Since $\fix(H)$ has at least $3$ points in it, the subgroup generated by $H$ and $fHf^{-1}$ cannot contain a loxodromic, and is therefore finite \cite[Ch. 8, \S 3]{GhH90}.  By maximality of $H$, we have that $f \in Q$.   This shows $Q$ is the stabilizer of $\fix(g)$. 

To conclude the proof we wish to show that $Q$ is quasi-convex and $\Lambda(Q) = \fix(g)$.   We can then take $Q_g = Q$ in the conclusion.
Let $C$ denote the quasi-convex hull of $\fix(g)$, meaning the set of all bi-infinite geodesics with both endpoints in $G$. 
Note that there is a uniform bound, say $r'$, on distance that any $h \in H$ can translate any point in $C$.

We now show $Q$ acts cocompactly on $C$, which is enough to show $Q$ has the desired properties.   
To see this, assume for contradiction that $C$ contains infinitely many distinct cosets $\{Q g_k\}$.  Since every $h \in H$ translates each $g_n$ a distance at most $r'$, the conjugates $g_k^{-1}Hg_k$ all lie in the $r'$ ball about the identity in $\Gamma$.  
It follows that for infinitely many pairs $i,j$ the subgroups $g_i^{-1}Hg_i = g_j^{-1}Hg_j$ agree, and thus $g_ig_j^{-1}\in Q$.  We conclude the cosets $\{Q g_k\}$ were not distinct, giving the desired contradiction.  
\end{proof} 
Recall that a \emph{null sequence} in a metric space is a collection of subsets $\mc{D}$ so that for every $\epsilon>0$, the set $\{D\in \mc{D}\mid \diam(D)>\epsilon\}$ is finite.
\begin{corollary}\label{cor:null_sequence}
  The collection of fixed point sets of nontrivial torsion elements is a null sequence in $\bgam$.
\end{corollary}
\begin{proof}
  Since $G$ is hyperbolic, it contains only finitely many conjugacy classes of torsion elements (see \cite[2.2 B]{GromovEssay}).  It therefore suffices to show the fixed point sets of elements in a single conjugacy class form a null sequence.  These fixed point sets are exactly the $G$--translates of $\fix(g)$ for some torsion element $g$.  
  By Lemma~\ref{lem:MO}, $\fix(g)$ is equal to $\Lambda(Q_g)$ for a quasi-convex subgroup $Q_g$ of $G$.  By Assumption~\ref{assump:faithful}, $g$ does not fix all of $\bgam$, so this subgroup $Q_g$ must be infinite index in $G$.
  By \cite[Corollary 2.5]{GMRS98}, the $G$--translates of $\Lambda(Q_g)$ form a null sequence.  
\end{proof}

\begin{proposition} \label{prop:empty_int}
$F$ is a closed set with empty interior. 
\end{proposition}

\begin{proof}
  Let $K \subset \triples$ be an arbitrary compact set.
  By Corollary~\ref{cor:null_sequence} there are only finitely many elements $g$ so that $\diam(\fix(g))\ge \minsep(K)$.  There are therefore only finitely many nontrivial $g$ so that $K$ contains a triple of points in $\fix(g)$.  For each such $g$, the set $\fix(g)\times\fix(g)\times\fix(g)$ intersects $K$ in a closed subset with empty interior in $\triples$.  It follows that $F\cap K$ is closed with empty interior.  Since $\triples$ can be exhausted by compact sets, the conclusion follows.
\end{proof} 

We also need the following general result about closed sets with empty interior in compact metric spaces. 

\begin{lemma} \label{lem:sparse_general}
  Let $A$ be a compact metric space and $C \subset A$ a closed subset with empty interior.  Given any $\epsilon >0$, there exists $\kappa>0$ so that $A \minus N_\kappa(C)$ is $\epsilon$--dense in $A$.
\end{lemma}

\begin{proof} 
  Let $d_A$ denote the metric on $A$.  Since $C$ has empty interior, for all $x \in A$ there exists a point $p_x \notin C$ with $d_A(p_x,x)<\epsilon$.  Since $A$ is compact, there is a finite collection $x_1,\ldots,x_k$ so that the open $\epsilon$--balls $B_\epsilon(p_{x_i})$ cover $A$.  The set $C$ is closed, so the distance $d_A(p_{x_i},C)$ is positive for each $i$.  We let $\kappa$ be half the minimum of the distances $d_A(p_{x_i},C)$.  Each $x\in A$ is contained in one of the balls $B_\epsilon(p_{x_i})$, so the set $\{p_{x_1},\ldots,p_{x_k}\}\subset A \minus N_\kappa(C)$ is $\epsilon$--dense.
\end{proof}

\begin{proof}[Proof of Proposition \ref{prop:F_sparse}]
  By Proposition~\ref{prop:empty_int},  $F$ is closed with empty interior.
  Recall from Proposition~\ref{prop:convergence_group} that $G$ acts properly discontinuously and cocompactly on $\triples$.
  Thus $F/G\subset \triples/G$ is closed with empty interior and the proposition follows immediately from Lemma~\ref{lem:sparse_general}, taking $A = \triples/G$ and $C = F/G$. 
\end{proof}

\subsection{Additional properties of $X_\kappa$}
We establish some properties of the sets $X_\kappa := \triples \minus N_\kappa(F)$.  First we make the following observation, relevant to the application of Propositon~\ref{prop:good_map}.
\begin{lemma}
  For any $\kappa>0$, the group $G$ acts properly, freely, and cocompactly by isometries on $X_\kappa$.
\end{lemma}
\begin{proof}
  The group $G$ already acts properly and cocompactly on $X$, and $N_\kappa(F)$ is open and $G$--invariant, so $G$ still acts properly and cocompactly on $X_\kappa$.
  The only points of $\triples$ with nontrivial stabilizer are in $F$, which has been removed.
\end{proof}

Next we give a technical refinement of Proposition~\ref{prop:F_sparse}, which is what we actually use in the next section.

\begin{lemma} \label{lem:theta_balls}
Suppose we are given a compact set $D \subset \triples$, and $\epsilon'' < \epsilon' < \minsep(D)/2$.    There exists $\kappa >0$ such that, 
for any $(a, \theta, c) \in D$, there exists $c'$ so that $\dvis(c, c') < \epsilon''$ and $\{a\} \times B_{\epsilon'}(\theta) \times \{c'\}$ lies in the interior of  $X_\kappa$.  
\end{lemma} 

We recall that $\minsep$ is defined in terms of the visual metric $\dvis$.  

\begin{proof}
  Let $D'$ be a compact subset of $\triples$ containing $B_{ \epsilon'}(x) \times B_{ \epsilon'}(y) \times B_{ \epsilon'}(z)$ for all $(x, y, z) \in D$.
  Let $F$ be the set of elements $g\in G - \{\idG\}$ so that $g$ fixes some triple in $D'$.  
  By Corollary~\ref{cor:null_sequence}, the fixed point sets of torsion elements form a null sequence in $\partial G$.  For each $g\in F$, we have $\diam(\fix(g))\ge \minsep(D') > 0$, so $F$ is finite.  Let $C = \bigcup\{\fix(g) \st g\in F\}$ be the union of these fixed sets in $\bgam$.  Since $C$ is a finite union of closed sets with empty interior, $C$ is closed with empty interior.
  
  By Lemma \ref{lem:sparse_general}, we can choose a positive $\lambda < \epsilon''$ such that $\bgam-N_\lambda(C)$ is $\epsilon''$--dense in $\bgam$.  Choose $\kappa$ small enough so that $N_{2\kappa}(F) \cap D'$ is a subset of the neighborhood of $F$ of radius $\lambda$ in the product visual metric on $(\bgam)^3$.  

  To verify this $\kappa$ works, choose $(a, \theta, c) \in D$.  Since $\bgam-N_\lambda(C)$ is $\epsilon''$--dense, we can find $c'\in \bgam \minus N_\lambda(C)$ with $\dvis(c, c')< \epsilon''$.   
  We claim that $\{a\} \times B_{\epsilon'}(\theta) \times \{c'\}$ is contained in $\triples \minus N_{2\kappa}(F)$, and hence in the interior of $X_\kappa$.   Indeed, for any $b \in B_{\epsilon'}(\theta)$, we have $(a,b,c') \in D'$.   Suppose $(x,y,z) \in F$ is a closest point of $F$ to $(a,b,c')$ in the product metric.  If the visual distance between each coordinate of $(x,y,z)$ and $(a,b,c)$ were less than $\lambda$, which is less than $\epsilon'$, then we would have $(x,y,z) \in D' \cap F$.  By our definition of $C$, this means that $x$, $y$, and $z$ all belong to $C$.  
  But the minimum distance from $c'$ to a point of $C$ is greater than $\lambda$ by construction, so we conclude $(a,b,c')$ lies outside the $\lambda$--neighborhood of $F$ in the product metric, and hence outside $N_{2\kappa}(F)$.  
\end{proof}

Finally, we need a technical result which will come into play at the very end of the proof, when we show our semi-conjugacy is well-defined.  To state it we need a definition.
\begin{definition}[$\kappa$ path property]
  We say a compact $D\subset \triples$ has the \emph{$\kappa$ path property} if the following holds:

  For each $b \in \bgam$, and pair of points $(a, b, c)$ and $(a', b, c')$ in $D$, there is a path $a_t$ from $a$ to $a'$ and $c'' \in \bgam$ such that $(a_t, b, c'') \in D \minus N_\kappa(F)$ for all $t$.
\end{definition}
\begin{lemma} \label{lem:paths}
For any compact set $K \subset \triples$, there exists $\kappa > 0$ and a compact set $D \supset K$ so that $D$ has the $\kappa$ path property.
\end{lemma} 

\begin{proof}[Proof]   
Fix a round metric $d_{\mathrm{rnd}}$ on $\bgam = S^n$.  The collection \[C_m:= \left \{(a, b, c) \in \triples \st d_{\mathrm{rnd}}(a,b)\geq \tfrac{1}{m},\ d_{\mathrm{rnd}}(a, c) \geq \tfrac{1}{m},\ d_{\mathrm{rnd}}(b,c) \geq \tfrac{1}{m}\right \}\] forms an exhaustion of $\triples$ by compact sets, so we may choose $m$ sufficiently large so that $K \subset C_m$ and the diameter of $S^n$ in the round metric is larger than $\frac{10}{m}$.   
 
Since $D$ is compact, only finitely many elements of $G$ fix a triple of points that meets $D$.   Let $C \subset S^n$ denote the union of these finitely many fixed sets; it is a closed subset of $\bgam$ with empty interior.    

By Lemma \ref{lem:sparse_general}, there exists $\lambda < \frac{1}{m}$ such that $S^n - N_\lambda (C)$ is $\frac{1}{m}$-dense in $S^n$ (again, using the round metric).  
Consider a shortest length geodesic path with respect to the round metric on $S^n$ between $a$ and $a'$.  If this path does not meet the closed $\frac{1}{m}$-ball about $b$, then call this path $a_t$.  Otherwise, modify this path by removing the segment that intersects the closed $\frac{1}{m}$-ball, replacing it with a path that lies on the boundary of this ball, and call the resulting path $a_t$.  (We are here using our Assumption~\ref{assump:notcircle} that $n \ge 2$.)  In either case, $a_t$ lies in the union of a radius $\frac{1}{m}$ ball and a geodesic segment.  

Since the diameter of $S^n$ in the chosen round metric is larger than $\frac{10}{m}$, we may find a point $\hat{c} \in S^n$ that avoids both the $\frac{2}{m}$-neighborhood of this path $a_t$ and the $\frac{2}{m}$ neighborhood of $b$.  Since $S^n - N_\lambda (C)$ is $\frac{1}{m}$--dense, we may move this point a distance of at most $\frac{1}{m}$ to find a point $c'' \notin N_\lambda(C)$.   Thus we have, $(a_t, b, c'') \in D$ for all $t$.  

Finally, choose $\kappa$ small enough so that any element of a triple in $N_\kappa(F) \cap D$ lies within the $\lambda$-neighborhood (in the round metric) of some point of $C$.    Note that this choice of $\kappa$ only depends on $D$ and $\lambda$, which both depended only on the set $K$.  
The point $c''$ was chosen so that $c'' \notin N_\lambda(C)$.  Thus, we conclude that $(a_t, b, c'') \notin N_\kappa(F)$ holds for all $t$, as desired.  
\end{proof} 

Note that even though $\dvis$ and the round metric give the same topology on $S^n$, the metric $\dvis$ may be rather strange-looking.  In particular $\dvis$ is not Riemannian, so even very small balls with respect to $\dvis$ may not be connected.  This is the reason for the use of the round metric in the proof above.

\section{Nearby representations give quasi-geodesic partitions}\label{sec:quasigeodesic}

In this section we fix a constant $\kappa$ to define a space $X_\kappa \subset \triples$ as in Section \ref{sec:Xkappa}, and fix a neighborhood of $\rho_0$.  Our goal is to show that every representation $\rho$ in this neighborhood has $\rho_0$ as a factor.  In Section \ref{ss:geodesicintersection} we show that representations in this neighborhood induce quasi-geodesic partitions of a section of $X_\kappa \times \bgam \to X_\kappa$.  The endgame of the proof, carried out in Section \ref{sec:endpoint}, consists of using the endpoints of these quasi-geodesics to build a semi-conjugacy.

\subsection{Fixing a neighborhood of $\rho_0$}\label{ss:neighborhood}
We first fix a neighborhood $V$ of the identity in the set of continuous self-maps $\bgam \to \bgam$.  By shrinking $V$ if necessary, we may assume that $V$ consists of degree one maps.
There is some $C_V>0$ so that any map $h$ satisfying the lower bound on Gromov products
\begin{equation}\label{eq:smallsemiconjugacy}\tag{$\dagger$}
  \forall  x \in \bgam,\quad   \gprod{h(x)}{x}{\idG}  > C_V 
\end{equation}
lies in $V$.  

In Definition~\ref{def:U(V)} we will specify a neighborhood $\mathcal U(V)$ of $\rho_0$ in $\Hom(G, \Homeo(\bgam))$, ultimately showing that any $\rho\in \mathcal U(V)$ is semiconjugate to $\rho_0$ by a map $h$ satisfying~\eqref{eq:smallsemiconjugacy}.
This is sufficient to prove Theorem \ref{thm:main}.

To specify $\mathcal U(V)$ and set up the proof, we need to fix several intermediate constants and compact sets.   Recall $\mathcal{S}$ denotes our chosen generating set for $G$.

\begin{notation} \label{notation:H}
  Let $H = \max\{2\delta,Q(3\delta)\}+1$, where $Q(\cdot)$ is from Lemma~\ref{lem:diambound}.
\end{notation} 

\begin{lemma}\label{lem:D01}
  There exist compact sets $D_0\subset D_{\frac{1}{2}} \subset D_1 \subset D_2$ in $\triples$ and positive constants $\epsilon_2 \le \epsilon_1$ and $\kappa_0$ satisfying the following.
  \begin{enumerate}
  \item $D_0$ contains $\pi^{-1}(\idG)$. (Hence $G D_0 = \triples$.)
  \item\label{itm:Dhalf} $D_{\frac{1}{2}} = \bigcup\{gD_0\st g\in \mathcal{S}\cup \mathcal{S}^{-1}\cup\{\idG\}\}$.
  \item\label{itm:ep1} $\epsilon_1 <\frac{1}{2}\minsep(D_{\frac{1}{2}})$.
  \item\label{itm:D1} $D_1$ contains $B_{\epsilon_1}(x)\times B_{\epsilon_1}(y)\times B_{\epsilon_1}(z)$ for every $(x,y,z)\in D_{\frac{1}{2}}$.
    \item $D_1$ contains $\pi^{-1}(B_R(\idG))$ where
    \begin{equation}
      \label{eq:R}\tag{$**$}
      R > \max \{ 24 H + 52\delta + \diam(\pi(D_0)),  C_V + 4H + 11\delta \}.
    \end{equation}
  \item \label{itm:slices} For all $\kappa < \kappa_0$, the $\kappa$ path property of Lemma \ref{lem:paths} holds for $D_1$,
  \item  $\epsilon_2 < \frac{1}{2}\minsep(D_1)$. 
  \item  $ D_2$  contains $B_{\epsilon_2}(x)\times B_{\epsilon_2}(y)\times B_{\epsilon_2}(z)$ for every $(x,y,z)\in D_1$.
  \end{enumerate}
\end{lemma}
The sets $D_0$, $D_1$, and $D_2$ play a major role in the rest of the proof.  The set $D_{\frac{1}{2}}$ will only appear in the last step.  

\begin{proof}
  We let $D_0$ be a compact subset of $\triples$ containing $\pi^{-1}(B_1(\idG))$ and define 
  $D_{\frac{1}{2}}$ as in item~\eqref{itm:Dhalf}.  We fix any positive $\epsilon_1$ satisfying item~\eqref{itm:ep1}.
  Let $K \subset \triples$ be a compact set large enough to contain the union of $\pi^{-1}(B_R(\idG))$, where $R$ is in equation \eqref{eq:R}, as well as the union of the sets
$B_{\epsilon_1}(x)\times B_{\epsilon_1}(y)\times B_{\epsilon_1}(z)$ for every $(x,y,z)\in D_{\frac{1}{2}}$. 
Applying Lemma \ref{lem:paths} to $K$ gives us a constant $\kappa_0$ and set $D_1$ containing $K$ so that the $\kappa_0$ path property holds for $D_1$; note that by definition the $\kappa$ path property remains true for any $\kappa < \kappa_0$, since $N_\kappa(F) \subset N_{\kappa_0}(F)$.    Then take any $\epsilon_2$ and $D_2$ as described.  

\end{proof}

\begin{lemma}
  \label{lem:epsilon}
  There exists positive $\epsilon <\frac{1}{2}\epsilon_2$ such that the following  hold:
\begin{enumerate}[label = ($\epsilon$\arabic*)]
\item\label{item:geodesicsmiss} Any geodesic between points $a,b$ with $\dvis(a,b)\le 2 \epsilon$ has distance at least $10\delta$ from $\pi(D_2)$ (and hence from $B_R(\idG)$).
\item\label{item:Bp} For every $p\in \bgam$, there is a closed contractible set $B_p$ so that
  \[ B_{\epsilon}(p) \subset B_p \subset B_{\epsilon_2}(p) .\]
\end{enumerate}
\end{lemma}
\begin{proof}
  That the first condition can be met is a consequence of compactness.  Indeed, by Lemma~\ref{lem:boundedK}, the $10 \delta$-neighborhood of the set $\pi(D_2)$ is bounded in $\Gamma$, hence contained in some compact ball centered at $\idG$.  Thus, there is a positive lower bound on the visual distance between the endpoints of any geodesic passing through that ball.
  
  We now address the second condition.  Since $S^n$ is compact, the visual and round metrics are uniformly equivalent.  Thus there is a $\mu>0$ so that the round ball of radius $2\mu$ about $p$ is contained in $B_{\epsilon_2}(p)$ for every $p\in S^n$, so we can take $B_p$ to be the closed round $\mu$--ball.
  There is then some $\epsilon>0$ so that $B_\epsilon(p)\subset B_p$ for every $p$.
\end{proof}

\begin{definition}[The neighborhood $\mathcal U (V)$]  \label{def:U(V)}Recall we fixed $C_V>0$ at the beginning of the subsection, see Equation \eqref{eq:smallsemiconjugacy}.    Using this $C_V$, fix the compact sets $D_0\subset D_{\frac{1}{2}} \subset D_1 \subset D_2 \subset \triples$ and positive constants $\epsilon < \epsilon_2 \le \epsilon_1$ and $\kappa_0$ satisfying the conclusions of Lemmas~\ref{lem:D01} and~\ref{lem:epsilon} above. 

Fix some $\kappa > 0$ satisfying 
  \begin{enumerate}
  \item $\kappa < \kappa_0$ and
    \item\label{item:kappa} For any $(a,c,\theta)\in D_1$, there exists $c'\in B_{\epsilon}(c)$ so that $\{a\}\times B_{\epsilon_2}(\theta)\times\{c'\}$ lies in the interior of $X_\kappa$.
  \end{enumerate}
  That such a $\kappa$ exists follows by applying Lemma~\ref{lem:theta_balls} with $D = D_1$, $\epsilon' = \epsilon_2$, and $\epsilon'' = \epsilon$.  Note that $\epsilon < \epsilon_2 < \minsep(D_1)/2$, so the hypotheses of that lemma are satisfied.
  
  Define $\mathcal U(V)$ to be a neighborhood of $\rho_0$ in $\Hom(G, \Homeo(\bgam))$ consisting of representations satisfying both of the following:
  
\begin{enumerate}[resume]
\item \label{item:smallgens} $\dvis(\rho(g),\rho_0(g))<\epsilon$ for every $g\in \mathcal{S}\cup \mathcal{S}^{-1}$; and 
\item \label{item:flat}the map $f^\rho$ defined in Section~\ref{sec:fdef}, taking $X = X_\kappa$, has the property that $f^\rho((D_2 \cap X) \times \{\theta\})\subset B_\epsilon(\theta)$ for every $\theta$.
\end{enumerate}  
\end{definition}

We now fix notation.  
\begin{convention} \label{convention:rho}
Fix some $\rho \in \mathcal U(V) $.  Since $\rho$ is fixed, we henceforth drop it from the notation, writing $f = f^\rho$.  
Since $\kappa$ is also fixed, we also drop it from our notation when convenient, writing $X$ for $X_\kappa$. 
\end{convention} 

We keep this convention for the remainder of the work.  Our eventual goal is to show that $\rho_0$ is a factor of $\rho$ via a semiconjugacy satisfying \eqref{eq:smallsemiconjugacy}.  
The reader will note that when $F = \emptyset$, then $X = \triples$.  (For example, this holds when $G$ is torsion-free.)
It may be helpful on a first reading to keep this special case in mind, thinking of $X$ as the space of distinct triples.

\subsection{A section partitioned into coarsely geodesic sets}\label{ss:geodesicintersection}

We begin by describing a natural foliation on a section of $\triples \times \bgam \to \triples$. 
Let $\sigma\from \triples \to (\triples \times \partial G)$ be the section given by $\sigma((a,b,c)) = ((a,b,c),b)$.
The image of $\sigma$ has a topological foliation by leaves 
\[L_a := \{ ((a,b,c), b) \st b, c \in \partial G - \{a\}, b \neq c \, \}.\]  
The leaves of this foliation are {\em coarsely geodesic} in the following sense.  For any $a \neq \theta$, the set  $L_a\cap (\triples \times \{\theta\})$ is the image in $\sigma(\triples)$ of a set whose projection to $\Gamma$ is Hausdorff distance at most $\delta$ from any geodesic joining $a$ and $\theta$. The set $L_a\cap (\triples\times\{a\})$ is empty.

We will show that the sets 
\[ L_a \cap f^\rho(X \times \{\theta\}) \subset \sigma(X)\]
 behave much like the sets  $L_a\cap (\triples \times \{\theta\})$, in the sense that for each $a$ and $\theta$ in $\partial G$, this set is either empty or looks ``coarsely geodesic,'' meaning the projection of each nonempty set $L_a \cap f^\rho(X \times \{\theta\})$ to the Cayley graph $\Gamma$ lies at a uniformly bounded Hausdorff distance from a bi-infinite geodesic in $\Gamma$.  

 To formalize this, define $\bar\pi$ on $\sigma(X)$ by $\bar\pi = \pi \circ \sigma^{-1}$.
 It follows from the definition of $\pi$ on sets that for $Z \subset \sigma(X)$, the set $\bar\pi(Z)$ is equal to the union $\bigcup_{z \in Z} \bar\pi(z)$.  Our goal for the section can now be restated:

\begin{proposition} \label{prop:quasi-geodesic}
For any $(a,\theta)\in \bgam^2$,
either $L_a \cap f^\rho(X \times \{\theta\})$ is empty, or $\bar\pi(L_a \cap f^\rho(X \times \{\theta\}))$ is uniformly bounded Hausdorff distance from a bi-infinite geodesic in $\Gamma$ with one endpoint equal to $a$.  
\end{proposition} 

In fact, we eventually prove something slightly stronger (Proposition~\ref{prop:gromov_prod}), but the statement and proof requires 
some some set-up. 

\begin{notation} \label{notation:Y}
For $\theta \in \bgam$, we let $Y_\theta = f(X \times \{\theta\}) \subset X \times \bgam$.  
For $a, c, \theta \in \bgam$, we let $f_{a,c,\theta}$ denote the map $x \mapsto f_\theta(a, x, c)$.  This map is defined on all points $x \in \bgam$ such that $(a, x, c) \in X$ and it is continuous on its domain of definition.  
\end{notation} 

Recall  $L_a \cap (\triples \times\{\theta\})$ is nonempty provided $a \neq \theta$.  Our first lemma ensures that $L_a \cap Y_\theta$ is nonempty provided $a$ and $\theta$ are far apart, as follows: 

\begin{lemma} \label{lem:fixed}
For any $(a, \theta, c) \in D_1$, the set $L_a \cap Y_\theta$ is nonempty, and contains a point of the form $((a, b_\theta, c'), b_\theta)$ where $b_\theta \in B_\epsilon(\theta)$, and $c' \in B_\epsilon(c)$; in particular we have $(a, b_\theta, c') \in D_2 \cap X$. 
\end{lemma}  

\begin{proof} 
By Item \eqref{item:kappa} of Definition~\ref{def:U(V)}, there exists a point $c' \in B_\epsilon(c)$ 
such that $\{a\}\times B_{\epsilon_2}(\theta)\times \{c'\} \subset X$.  
 Thus $f_{a,c',\theta}$ is defined on $B_{\epsilon_2}(\theta)$ and 
Definition~\ref{def:U(V)} item~\eqref{item:flat} implies that $f_{a,c',\theta}(B_{\epsilon_2}(\theta)) \subset B_\epsilon(\theta)$.   Lemma~\ref{lem:epsilon} states that there exists a closed, contractible set $B_\theta$ such that 
   \[ B_{\epsilon}(\theta) \subset B_\theta \subset B_{\epsilon_2}(\theta) .\]
   Thus,
   \[f_{a,c',\theta}(B_\theta) \subset f_{a,c',\theta}(B_{\epsilon_2}(\theta))\subset B_\epsilon(\theta) \subset B_\theta.\] By the Lefschetz fixed point theorem, $f_{a,c',\theta}$ has a fixed point $b_\theta\in B_{\epsilon}(\theta)$, which means exactly that $((a, b_\theta, c'), b_\theta) \in L_a \cap Y_\theta$.  Recall that $\epsilon < \epsilon_2$, so by our choice of sets and constants as in Lemma \ref{lem:D01}, we have also $(a, b_\theta, c') \in D_2$.  
\end{proof}

\begin{notation}
Going forward, we define 
  $S(a,\theta) := \bar \pi (L_a\cap Y_\theta) \subset \Gamma$ 
\end{notation}
Our goal is to show that, whenever $S(a,\theta)$ is nonempty, it is bounded
Hausdorff distance from a bi-infinite geodesic with one endpoint equal to $a$.   The following lemma gives a local estimate.  
 Recall we have fixed $H = \max\{2\delta,Q(3\delta)\}+1$, where $Q(\cdot)$ is the function from Lemma~\ref{lem:diambound}.
 
\begin{lemma}
  \label{lem:localquasi}
  Suppose there is a point $((a,b_0,c_0), b_0)$ in $\sigma(D_1)\cap L_a\cap Y_\theta$.  
  Let $\gamma$ be a geodesic from $a$ to $b_0$ in $\Gamma$.  Then
  \[ S(a,\theta)\cap B_R(\idG)\subset N_H(\gamma),\mbox{ and } \gamma\cap B_R(\idG)\subset N_H(S(a,\theta)).\]
\end{lemma} 
 
\begin{proof}
  For the first inclusion, suppose that $p$ lies in $S(a,\theta)\cap B_R(\idG)$.  Then, since $D_1$ contains $\pi^{-1}(B_R(\idG))$,  we have $p \in \overline{\pi}(f((D_1 \cap X)\times\{\theta\})\cap L_a)$.  In particular, there exist $b,c\in \bgam$, such that all the following hold simultaneously: 
  \begin{enumerate}
  \item $p\in \pi(a,b,c)$,
  \item $(a,b,c)\in D_1 \cap X$,
  \item $f_\theta(a,b,c) = b$.    
  \end{enumerate}
  Both $(a,b,c)$ and $(a,b_0,c_0)$ lie in $(D_1 \cap X) \subset (D_2 \cap X)$.  By item \eqref{item:flat} in Definition \ref{def:U(V)}, we conclude 
  that $b = f_\theta(a,b,c)$ and $b_0 = f_\theta(a,b_0,c_0)$ both lie in $B_\epsilon(\theta)$.  In particular $\dvis(b,b_0) < 2\epsilon$, so any geodesic from $b$ to $b_0$ misses $B_R(\idG)$ by at least $10\delta$ (Lemma~\ref{lem:epsilon}.\ref{item:geodesicsmiss}).  By $\delta$--slimness of ideal triangles this implies
  that if $[a,b]$ is any geodesic joining $a$ to $b$, then $[a,b]\cap B_{R+8\delta}(\idG)$ lies in the $\delta$--neighborhood of $\gamma$ (the geodesic joining $a$ to $b_0$).  By definition of $\pi$, the point $p$ lies within $\delta$ of a point on $[a,b]$ which is at most $R+\delta$ from $\idG$, so lies within $2\delta$ of $\gamma$.  This shows the first inclusion.

  For the second inclusion suppose that $p\in \gamma\cap B_R(\idG)$.  By assumption~\ref{item:ontoline} on $\delta$, there is some $c$ so that $p\in \pi(a,b_0,c)$.  Since $D_1$ contains $\pi^{-1}(B_R(\idG))$, we have $(a,b_0,c)\in D_1$. 
  
By Item \eqref{item:kappa} of Definition~\ref{def:U(V)}, there exists $c' \in B_\epsilon(c)$ such that 
 \[ \{a\}\times B_{\epsilon_2}(b_0)\times\{c'\} \subset (X \cap D_2).\]  
Let $B = B_\theta$ be the contractible set from Lemma~\ref{lem:epsilon}.\ref{item:Bp}.  This set contains $B_\epsilon(\theta)$, so in particular $b_0\in B$.  We have $\{a\}\times B\times \{c'\}\subset X \cap D_2$, so
as in the proof of Lemma~\ref{lem:fixed},
$f_\theta$ induces a map $f_{a,c',\theta}\from B\to B_{\epsilon}(\theta)\subset B$, and $f_{a,c',\theta}$ fixes some $b'\in B_\epsilon(\theta)\subset B_{2\epsilon}(b_0)$.  The point $((a,b',c'),b')$ is therefore in $L_a\cap Y_\theta$, so $\pi(a,b',c')$ is a subset of $S(a,\theta)$.  Let $q\in \pi(a,b',c')$.
\begin{figure}[htbp]
  \labellist 
  \hair 2pt
\pinlabel $b_0$ at 315 260
\pinlabel $b'$ at 290 290
\pinlabel $c$ at 100 315
\pinlabel $c'$ at 80 300
\pinlabel $\gamma$ at 285 230
\pinlabel $q$ at 165 175
\pinlabel $p$ at 225 175
\pinlabel $a$ at 73 55
\pinlabel $B_R(\idG)$ at 195 100
   \endlabellist
   \centerline{ \mbox{
 \includegraphics[width = 2.3in]{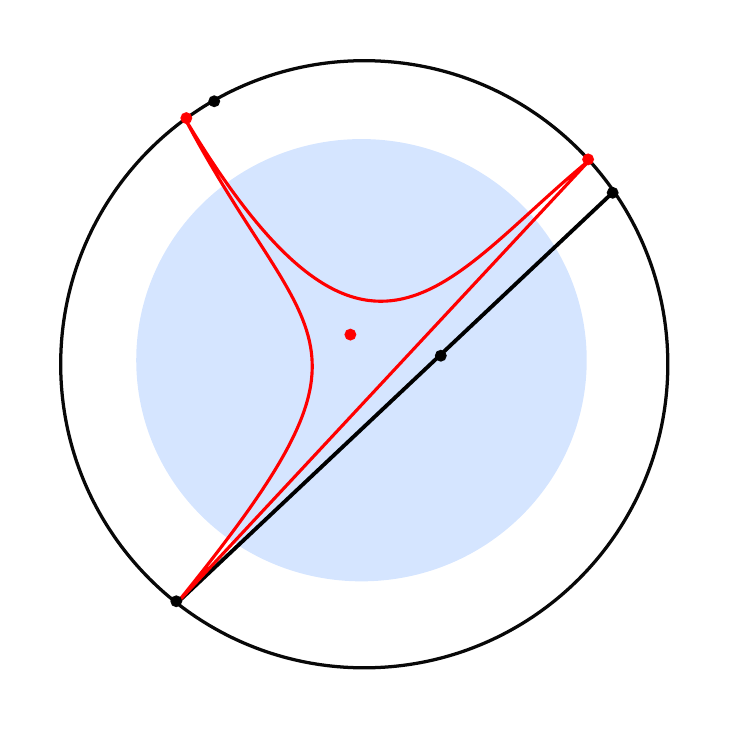}}}
 \caption{$p \in \pi(a,b,c)$ is close to any point $q$ of $\pi(a, b', c')$}
  \label{fig:neighborhood}
\end{figure}
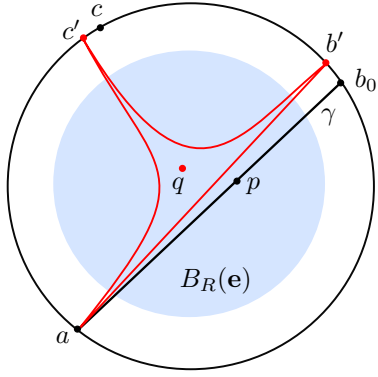 
See Figure~\ref{fig:neighborhood}.

 We claim $q$ is close to $p$. 
To see this, we first show $p$ is within $2\delta$ of any geodesic from $a$ to $b'$. 
Fix such a geodesic $[a,b']$, and geodesics $[b_0, b']$ and $[b_0, a]$.  Since $p \in B_R(\idG)$, it lies distance at least $10 \delta$ from any point on $[b_0, b]$.  Also we have that $p$ lies within $\delta$ of some point on $[b_0,a]$, so by $\delta$-thinness, $p$ is distance at most $2\delta$ from a point on $[a,b']$.   

Repeating this argument with $c'$ in place of $b'$ shows that $p$ is within $2\delta$ of any geodesic from $a$ to $c'$.  A slight modification shows that $p$ is within $3\delta$ of any geodesic $[b',c']$, as follows:  Considering a quadrilateral with sides $[b_0, c]$, $[c, c']$, $[b', c']$ and $[b_0, b']$, we know that $p$ is within $\delta$ of some point on $[b_0, c]$ and this must lie within $2\delta$ of a point on $[b', c']$ since the other two sides are each distance at least $10\delta$ from $p$.  
Thus, $p \in \pi_{3\delta}(a,b',c')$, and of course $q$ lies in this set as well. 
 
 By Lemma~\ref{lem:diambound}, $\pi_{3\delta}(a,b,c)$ has diameter at most $Q(3\delta)$.  In particular $\dgam(p,q)\le Q(3\delta)$.  Since $q\in S(a,\theta)$, this shows the second inclusion.
\end{proof}

We will now combine this work with Lemma~\ref{lem:bi-infinite} to prove Proposition~\ref{prop:quasi-geodesic}.  We will actually prove the following slightly stronger statement.  

\begin{proposition} \label{prop:gromov_prod}
  If $S(a,\theta) \neq \emptyset$, then  $S(a,\theta)$ is Hausdorff distance less than $3H+6\delta+1$ from a geodesic $\gamma$ with one endpoint at $a$.

  If in addition $L_a\cap Y_\theta\cap \sigma(D_0)$ is non-empty, then $\gamma$ joins $a$ to a point $e^+$ so that $\gprod{e^+}{\theta}{\idG} \ge R-(4H + 11\delta)$. 
\end{proposition}
\begin{proof}
  We first verify the hypotheses of Lemma~\ref{lem:bi-infinite} for
  the set $S = S(a,\theta)\cap G$.  (Recall that $G$ is canonically identified with the vertices of $\Gamma$.)  By Corollary~\ref{cor:deltaOK}, $\pi$ is surjective, so for each $s\in \mathcal{S}$ we can choose some $b_s, c_s$ so that $\sigma(a,b_s,c_s) \in L_a\cap Y_\theta$ and $s\in \pi(a,b_s,c_s)$.   Since $\pi^{-1}(\idG)\subset D_0$, the point $\sigma(a,b_s,c_s)$ lies in $\sigma(s\cdot D_0)\cap L_a\cap Y_\theta$.  Left translating by $s^{-1}$, we have 
  \[ (s^{-1}a,s^{-1}b_s,s^{-1}c_s)\in \sigma(D_0)\cap L_{s^{-1}a}\cap Y_{\rho(s^{-1})\theta}.\]  
  Let $\hat{\gamma}_s$ be any bi-infinite geodesic from $s^{-1}a$ to $s^{-1}b_s$.  Lemma~\ref{lem:localquasi} implies
  \[ s^{-1}S\cap B_R(\idG)\subset N_H(\hat{\gamma}_s),\mbox{ and } \hat{\gamma}_s\cap B_R(\idG)\subset N_H(S).\]
  Setting $\gamma_s$ equal to $s \cdot \hat{\gamma}_s$, we find that
  \[ S\cap B_R(s)\subset N_H(\gamma_s),\mbox{ and } \gamma_s\cap B_R(s)\subset N_H(S).\]

  Now fix any $\frac{R}{4}$--connected component $S_0$ of $S$ so that we may apply Lemma~\ref{lem:bi-infinite}.  (We will see later that $S$ is $\frac{R}{4}$--connected, so in fact $S_0 = S$.)  The paragraph above shows that all the hypotheses of Lemma~\ref{lem:bi-infinite} hold for $S_0$ so we conclude that there is a bi-infinite geodesic $\gamma$ with $\dHaus(\gamma,S_0) \le 3H + 6\delta$.

  We next claim that one of the endpoints of $\gamma$ is $a$.  We argue by contradiction.  Using $\delta$--slimness of ideal triangles, there is a point $p\in \gamma$ so that $p$ is within $2\delta$ of any geodesic joining $a$ to any endpoint of $\gamma$.  Since $S_0$ is Hausdorff distance at most $3H + 6\delta$ from $\gamma$, there is some $s\in S_0$ so that $\dgam(s,p)\le 3H + 6\delta$.  Now $\gamma_s$ has an endpoint at $a$, and we have $\gamma_s\cap B_R(s)\subset N_H(S)$.   
  From the inclusions $S\cap B_R(s)\subset N_H(\gamma_s),\mbox{ and } \gamma_s\cap B_R(s)\subset N_H(S)$ and the inequality $6H < \frac{R}{4}$, we conclude that 
  $N_H(\gamma_s) \cap S \subset S_0$.  Choosing $x\in \gamma_s$ at distance $\frac{R}{2}$ from $s$, in the direction of $a$, we find some point $s'\in S_0$ with $\dgam(x,s')\le H$.  See Figure~\ref{fig:estimates} for a schematic.  By repeated applications of the triangle inequality, one can easily show that this $s'$ is further than $3H + 6\delta$ from $\gamma$, a contradiction.

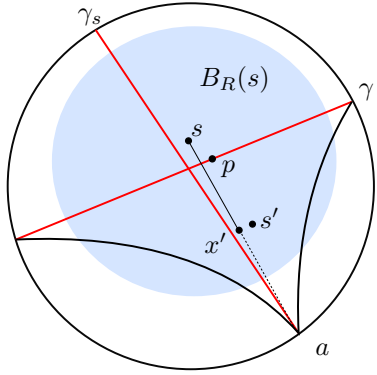
\begin{figure}
  \labellist 
  \hair 2pt
\pinlabel $\gamma$ at 315 250
\pinlabel $\gamma_s$ at 95 310
\pinlabel $s$ at 180 220
\pinlabel $p$ at 205 189
\pinlabel $s'$ at 238 150
\pinlabel $a$ at 280 45
\pinlabel $x'$ at 195 130
\pinlabel $B_R(s)$ at 210 260
   \endlabellist
   \centerline{ \mbox{
 \includegraphics[width = 2.3in]{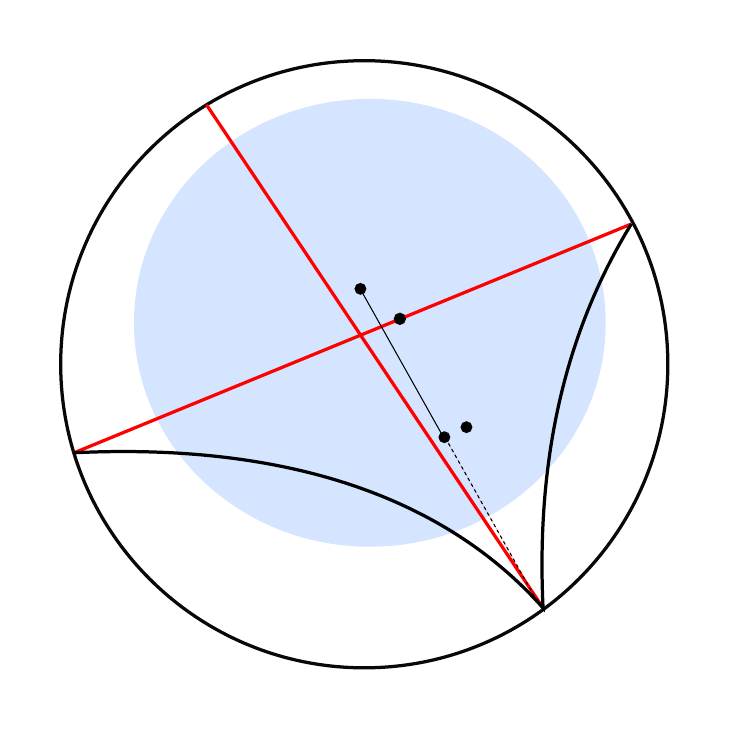}}}
 \caption{$\gamma$ and $\gamma_s$ should both be close to $S_0$ on the shaded region $B_R(s)$, giving a contradiction if $a$ is not an endpoint of $\gamma$.}
  \label{fig:estimates}
\end{figure}

We now argue that $S = S_0$.  To see this, suppose that $S_1$ were some other component.  We may apply the same argument to $S_1$ to produce a bi-infinite geodesic $\gamma_1$.  The geodesics $\gamma$ and $\gamma_1$ share an endpoint $a$, and so contain points within $\delta$ of one another.  This implies that $S_0$ and $S_1$ contain points within $6H + 13\delta < \frac{R}{4}$ of one another, so they cannot be different $\frac{R}{4}$--connected components.  

  We have established the first conclusion, since $\dHaus(S,S(a,\theta))\le 1$.

  Now we suppose that $s\in S$ is in $\bar\pi(\sigma(D_0)\cap L_a\cap Y_\theta)$.  By Lemma~\ref{lem:fixed},  we may take $s$ to be in $\bar \pi((a,b,c),b)$ for some $b\in B_\epsilon(\theta)$.  We can therefore take $\gamma_s$ in the first part of this argument to be a geodesic joining $a$ to $b$.  The second conclusion of Lemma~\ref{lem:bi-infinite} implies that $\gprod{e^+}{b}{s}\ge R-(4H + 10\delta)$, where $e^+$ is the endpoint of $\gamma$ which is not equal to $a$.  We thus have
  \begin{align*}
    \gprod{e^+}{b}{\idG} & \ge \gprod{e^+}{b}{s} - \dgam(\idG,s)\\
                      & \ge \gprod{e^+}{b}{s} - \diam (\pi(D_0)) \\
                      & \ge R - (4H + 10\delta)
  \end{align*}

  The first condition on $\epsilon$ in Lemma~\ref{lem:epsilon} implies that $\gprod{b}{\theta}{\idG}\ge R$.  We have
  \begin{align*}
    \gprod{e^+}{\theta}{\idG}& \ge \min\{ \gprod{e^+}{b}{\idG}, \gprod{b}{\theta}{\idG}\} - \delta\\
                          & \ge R - (4H + 11\delta),
  \end{align*}
  establishing the last claim of the Proposition.
\end{proof}

\section{The endpoint map} \label{sec:endpoint}
To summarize the results of the previous section, for each pair $a, \theta \in \partial G \times \partial G$ so that $L_a \cap Y_\theta \neq \emptyset$, there is a geodesic in $\Gamma$ at (uniformly) bounded Hausdorff distance from $S(a,\theta) = \overline{\pi} (L_a \cap Y_\theta)$, with one endpoint equal to $a$.  Say this geodesic is {\em shadowed} by $S(a, \theta)$, and orient it so that the negative endpoint is $a$.  Any two bi-infinite geodesics shadowed by $S(a,\theta)$ are bounded Hausdorff distance from each other, so they have the same endpoints in $\bgam$.
This gives us positive and negative ``endpoint maps'' $e^+$ and $e^-$ assigning to each pair $(a, \theta)$ where $L_a \cap Y_\theta \neq \emptyset$ the positive and negative endpoints of the shadowed geodesic.  For any such $(a,\theta)$, we have $e^-(a, \theta) = a$ and $e^+(a, \theta) \neq a$.     
Furthermore, the equivariance property in Proposition~\ref{prop:good_map} implies that for any $g \in G$, we have 
\[ g S(a,\theta) = S(g a, \rho(g) \theta ) = \overline{\pi} (L_{ga} \cap Y_{\rho(g) \theta}). \]
This implies the following equivariance of the positive endpoint map.
\begin{equation} \label{eq:equivariance}
 g \cdot e^+(a, \theta) = e^+(ga, \rho(g) \theta). 
\end{equation}
Here and in what follows, we will frequently omit $\rho_0$ from the notation when it is clear that we are referring to the natural action of $G$ on its boundary.  Thus, we will write $ga$ rather than $\rho_0(g)a$ when $a$ is a boundary point.  For $x \in X$, we also write $gx$ for its image under the standard action of $g \in G$ on $X \subset \triples$.  

We will first establish continuity of the positive endpoint map on a large set, then use it to define a semi-conjugacy.

\subsection{Continuity}

\begin{proposition}[Continuity in $a, \theta$ over $D_1$] \label{prop:cont_endpt}
Suppose $a_0, \theta_0 \in \bgam \times \bgam$ is such that there exists $c$ with $(a_0, \theta_0, c) \in D_1$.  Then the map $(a, \theta) \mapsto e^+(a, \theta)$ is continuous at $(a_0, \theta_0)$.  
\end{proposition} 

Proposition~\ref{prop:cont_endpt} will be a quick consequence of the following technical lemma.
\begin{lemma} \label{lem:cont}
  Suppose $L_{a_0} \cap Y_{\theta_0}$ is nonempty.  For any $r>0$, there is a neighborhood $N = N(r)$ of $(a_0,\theta_0)\in \bgam \times \bgam$ so that if $(a , \theta) \in N$, then $B_r(\idG)\cap S(a_0,\theta_0)$ lies in the $\diam(\pi(D_1))$--neighborhood of $S(a,\theta)$. 
\end{lemma}

The main idea behind the proof of Lemma~\ref{lem:cont} comes from the proof of Lemma~\ref{lem:fixed}, and the fact that the fixed point property used there is stable under small perturbations of the map $f_{a,c,\theta}$.

\begin{proof}[Proof of Lemma~\ref{lem:cont}]
  If $B_r(\idG)\cap S(a_0,\theta_0)$ is empty there is nothing to show, so we suppose
  $B_r(\idG)\cap S(a_0,\theta_0)$ is non-empty.  Let $K\subset X$ be the closure of $\pi^{-1}(B_r(\idG))$.
  
Recall from Notation \ref{notation:Y} that $f_{a,c,\theta}$ denotes the map $x \mapsto f_\theta(a, x, c)$, and recall that $D_0\subset D_1$ has the following properties
\begin{enumerate}
\item $X \subset GD_0$
\item   For any $(a,b,c)\in D_0$, the set $B_{\epsilon_1}(a) \times B_{\epsilon_1}(b) \times B_{\epsilon_1}(c)$ is contained in $D_1$.  (See Lemma~\ref{lem:D01}.\eqref{itm:D1}.)
\end{enumerate}
Recall also that $\epsilon_1 >  \epsilon_2 > 2 \epsilon$.

We cover $\sigma(K) \cap (L_{a_0} \cap Y_{\theta_0})$ with translates of $D_0$, as follows.  Since $K$ is compact, there are finitely many elements $g_1, g_2, \ldots g_k \in G$ so that 
\[ \sigma(K) \cap (L_{a_0} \cap Y_{\theta_0}) \subset \bigcup_{i=1}^k \sigma(g_i D_0). \]
By deleting elements from the list if necessary we may assume 
\[ \sigma(g_i D_0) \cap \left(L_{a_0} \cap Y_{\theta_0} \right) \neq \emptyset\] 
for each $i$. 

Our next goal is to show that, for each $i$, the projection of the larger translate $g_iD_1$ to the Cayley graph contains a point of $S(a, \theta)$, provided that $(a, \theta)$ is chosen close enough to $(a_0,\theta_0)$.  Here ``close enough'' depends on the set $K$ and hence on the constant $r$. 

Translating back to $D_0$, for each $i$ we have $\sigma(D_0) \cap \left( L_{g_i^{-1}a_0} \cap Y_{\rho(g_i)^{-1}\theta_0} \right) \neq \emptyset$.  
Let $b_i, c_i \in \bgam$ be such that 
\[ \left( ( g_i^{-1}a_0, b_i, c_i ), b_i \right) \in \sigma(D_0) \cap \left( L_{g_i^{-1} a_0} \cap Y_{\rho(g_i)^{-1}\theta_0} \right). \]
Because $ \left( g_i^{-1}a_0, b_i, c_i \right) \in D_0 \subset D_2$, Definition~\ref{def:U(V)} \eqref{item:flat} implies that 
\[ b_i = f_{\rho(g_i)^{-1}\theta_0}(g_i^{-1}a_0, b_i, c_i) \in B_{\epsilon}(\rho(g_i)^{-1}\theta_0). \]
hence, $\dvis(b_i,\rho(g_i)^{-1}\theta_0) < \epsilon$, and so $(g_i^{-1}a_0, \rho(g_i)^{-1}\theta_0, c_i ) \in D_1$.  

Let ${p_i} = \rho(g_i)^{-1}\theta_0$.  By Item~\ref{item:kappa} of Definition~\ref{def:U(V)}, there is some s $c'_i \in B_\epsilon(c_i)$ so that 
$ \{a_i\} \times B_{\epsilon_2}({p_i}) \times \{c'_i\}$ lies in the interior of  $X$.
Furthermore by Lemma~\ref{lem:epsilon}.\eqref{item:Bp} there is a 
closed contractible set $B_{p_i} \subset \bgam$ with $B_\epsilon({p_i}) \subset B_{p_i} \subset B_{\epsilon_2}({p_i})$, and 
 \[ f_{g_i^{-1} a_0, c'_i, {p_i}}(B_{\epsilon_2}({p_i})) \subset B_\epsilon({p_i}),\]
 so $f_{g_i^{-1} a_0, c'_i, {p_i}}$ has a fixed point in $B_\epsilon({p_i})$.  
 The property of taking the compact set $B_{p_i}$ into the open ball $B_\epsilon({p_i})$ is open (in the compact-open topology on continuous maps), so also holds for 
 any map sufficiently close to $f_{g_i^{-1} a_0, c'_i, {p_i}}$, provided the map is defined on $B_{p_i}$.   Recall that the domain of definition of $f_{x,y,z}$ is the set $\{ w \st (x, w, z) \in X \}$.  Thus, if a function $f_{x,y,z}$ is defined on a set $\{x\} \times B_{\epsilon_2}({p_i}) \times \{z\}$ contained in the {\em interior} of $X$, and 
  $B_{\epsilon_2}({p_i}) \supset B_{p_i}$, then for all sufficiently close $x', y', z'$ the function $f_{x',y', z'}$ will be defined on $B_{p_i}$ as well.  
 Additionally, the functions $f_{x,y,z}$ vary continuously in the arguments $(x,y,z)$.  Thus, we may take a neighborhood $N_i$ of $(a_0, \theta_0)$ such that for each $(a, \theta) \in N_i$,
 \begin{enumerate}
 \item the map $f_{g_i^{-1} a, c'_i, \rho(g_i)^{-1}\theta}$ is defined on $B_{p_i}$
  \item the map $f_{g_i^{-1} a, c'_i, \rho(g_i)^{-1}\theta}$ has a fixed point contained in $B_\epsilon({p_i})$; and
  \item $\dvis(g_i^{-1} a,g_i^{-1} a_0)< \epsilon_1$.
 \end{enumerate}

 Set $N = \bigcap_{i=1}^k N_i$.  Thus,  for any $(a, \theta) \in N$, each of the sets $L_{g_i^{-1}(a)} \cap Y_{\rho(g_i)^{-1}\theta}$ contains a point $\sigma (g_i^{-1} a, z_i, c'_i)$ where $z_i \in B_\epsilon({p_i})$.  In particular, we have 
 \begin{align*}
& \dvis(g_i^{-1} a, g_i^{-1} a_0)< \epsilon_1 \text{ and } \\
& \dvis(z_i,b_i)\le 2\epsilon <\epsilon_1.
\end{align*}  
Since $(g_i^{-1} a_0,b_i,c_i)\in D_0$, this means that $(g_i^{-1} a, z_i, c'_i)\in D_1$.  Multiplying on the left by $g_i$, we obtain $(a,g_iz_i,g_ic'_i) \in g_iD_1$ and $\sigma(a,g_iz_i,g_ic'_i)\in L_{a}\cap Y_{\theta}$, so the intersection
 \[ \sigma(g_i D_1) \cap (L_a\cap Y_\theta)  \]
 is non-empty for each of the elements $g_i$.  Projecting to the Cayley graph, we have $B_r(\idG)\cap S(a_0,\theta_0)$ contained in the $\diam(\pi(D_1))$--neighborhood of $S(a,\theta)$, which proves the lemma.   
\end{proof} 

\begin{proof}[Proof of Proposition~\ref{prop:cont_endpt} from Lemma~\ref{lem:cont}]
 Suppose $(a_0, \theta_0) \in \bgam \times \bgam$ is such that there exists $c$ with $(a_0, \theta_0, c) \in D_1$.  Then by Lemma~\ref{lem:fixed}, $L_{a_0} \cap Y_{\theta_0}$ contains a point of $\sigma(D_2)$, so is nonempty.
 Lemma~\ref{lem:cont} states that, given $r>0$, there is a neighborhood $N$ of $(a_0,\theta_0)$ so that if $(a , \theta) \in N$, then $B_r(\idG)\cap S(a_0,\theta_0)$ lies in the $\diam(\pi(D_1))$--neighborhood of $S(a,\theta)$.  By Proposition~\ref{prop:gromov_prod} both $S(a_0,\theta_0)$ and $S(a,\theta)$ are Hausdorff distance at most $3H+6\delta+1$ from some bi-infinite geodesic, so these geodesics will $2(3H + 6\delta + 1)+ \diam(\pi(D_1))$ fellow-travel each other over a compact set, which can be taken as large as we wish by taking $r$ large.  This gives continuity.   
 \end{proof}

 We now establish continuity of a similar positive endpoint map defined everywhere on $X$.   Since $f$ is a homeomorphism and the sets $X\times\{\theta\}$ partition $X\times\bgam$, so do their images $Y_\theta$.  Likewise, the sets $L_a$ give a partition of $\sigma(X)\subset X\times \bgam$, so 
  for each $x \in X$, there exists a unique $a(x)$ and $\theta(x)$ such that $\sigma(x) \in L_{a(x)} \cap Y_{\theta(x)}$.  Note if $x = (a,b,c)$ then $a(x) =a$.  
We have a sequence of maps: 
\begin{equation}\label{eq:curlyE}
  x = (a,b,c) \mapsto ((a, b, c), b) \mapsto (a, \theta(x)) \mapsto e^+(a, \theta(x)).
\end{equation}

\begin{proposition}[Continuity on $X$] \label{prop:cont_X}
Let $\mathcal{E}^+(x)= e^+(a(x), \theta(x))$ be the map given by the composition in~\eqref{eq:curlyE}.  Then $\mathcal{E}^+$ is continuous on all of $X$. 
\end{proposition}

\begin{proof}
  Equivariance of each map in the composition implies that we have the equivariance property
  \[ \mathcal{E}^+(g x) = \mathcal{E}^+(g a(x),\rho(g)\theta(x)) .\]
It therefore suffices to check continuity of the composition above on the set $D_0 \cap X$ containing a fundamental domain for the action of $G$ on $X$.  

By definition, the section $\sigma$ is continuous.  For the second map in \eqref{eq:curlyE}, note that $\theta(x)$ is simply projection onto the second coordinate of $f^{-1}(\sigma(x)) \in X \times \bgam$.  Since $f^{-1}$ is a homeomorphism of $X \times \bgam$, its projection $\theta(x)$ is continuous.  Finally, Proposition~\ref{prop:cont_endpt} says that $(a, \theta(a,b,c))  \mapsto e^+(a, \theta)$ is continuous for $(a,b,c) \in D_1$.  
\end{proof} 
Going forward we will abuse notation and often think of $\mathcal{E}^+$ as a map defined on $\sigma(X)$, via the identification of $X$ with $\sigma(X)$.  

\subsection{Proof of Theorem~\ref{thm:main}}
Using the work above, we may now conclude the proof of our main theorem.  First recall the statement.
\maintheorem*

Our neighborhood $U = \mathcal U(V)$ was determined by our desired lower bound $C_V$ on Gromov products when we set our conventions in Section \ref{ss:neighborhood}.  In this section, we show that $e^+(a,\theta)$ is (locally) a function only of $\theta$, hence can be thought of a map from $\bgam$ to $\bgam$.  We will then show that this map has the properties of the desired semi-conjugacy between $\rho_0$ and $\rho$.

\begin{lemma} [$e^+$ is locally a function of $\theta$] \label{lem:only_theta}
Let $\theta \in \bgam$ and let $\{a_t\st t \in [0,1]\}$ be a path in $\partial G$ so that $e^+(a_t, \theta)$ is defined and continuous at all points.  
Then $e^+(a_t, \theta)$ is constant.   
\end{lemma} 

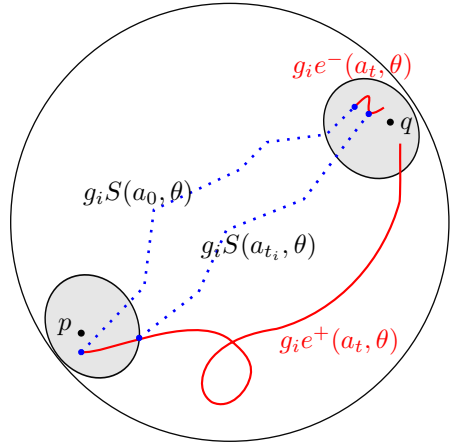
\begin{figure}
  \labellist 
  \hair 2pt
\pinlabel \textcolor{red}{$g_ie^-(a_t,\theta)$} at 270 295
\pinlabel $g_iS(a_0,\theta)$ at 110 200
\pinlabel $p$ at 55 100
\pinlabel $q$ at 310 250
\pinlabel \textcolor{red}{$g_ie^+(a_t,\theta)$} at 260 90
\pinlabel $g_iS(a_{t_i},\theta)$ at 200 160
   \endlabellist
   \centerline{ \mbox{
 \includegraphics[width = 2.5in]{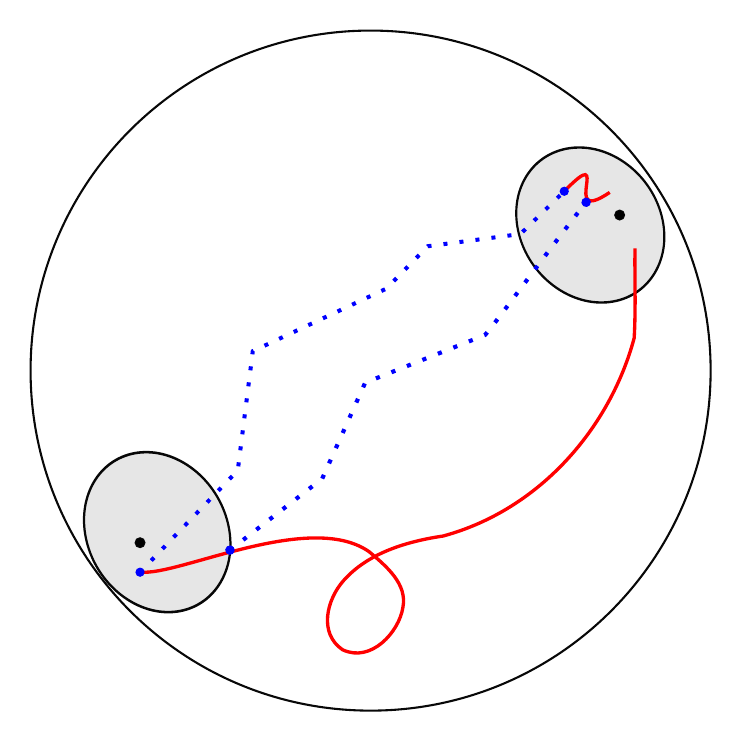}}}
 \caption{Paths of endpoints (in red) on $\bgam = \partial \Gamma$ and associated near-geodesic sets in $\Gamma$ (blue) with endpoints in the $\epsilon$-balls about $p$ and $q$.}
  \label{fig:onlytheta}
\end{figure}

\begin{proof} 
  We argue by contradiction.  Suppose we have such a path where $e^+$ is nonconstant.  Truncating and reparameterizing, we may suppose that
  \begin{enumerate}
  \item $e^+(a_t,\theta)$ is not locally constant at $t=0$,
  \item for all $t\in [0,1]$, $a_t \neq e^+(a_0,\theta)$, and
  \item $e^+(a_1,\theta)\ne e^+(a_0,\theta)$.
  \end{enumerate}
  The second item can be ensured by taking any sufficiently short path that is nonconstant at $0$, since $e^+(a, \theta) \neq a$ holds for all $a, \theta$.
  
  Since $G$ acts on $\partial G$ as a uniform convergence action (see Proposition~\ref{prop:convergence_group}), the point $e^+(a_0,\theta)$ is a conical limit point so there exists a sequence $\{g_i\} \subset G$ and points $p \neq q \in \partial G$ such that $g_i e^+(a_0,\theta) \to p$ and $g_i z \to q$ for all $z \in \partial G - \{e^+(a_0,\theta)\}$.    Modifying the sequence $\{g_i\}$ if needed by postcomposing with some fixed $g \in G$, we may also assume there exists $c$ with $(p,q,c) \in D_0$.  In particular this implies that $B_\epsilon(p)\cap B_\epsilon(q) = \emptyset$.
  Since $\bgam$ is compact we may assume by passing to a subsequence that $\rho(g_i)\theta$ converges to some point $\theta_\infty \in \partial G$.

For $i$ large enough, $g_i \cdot e^+(a_1, \theta) \in B_{\epsilon/2}(q)$ and $g_i \cdot e^+(a_0, \theta) \in B_{\epsilon/2}(p)$.  Thus, for each sufficiently large $i$, there exists $t_i$
such that $p_i:= g_i \cdot e^+(a_{t_i}, \theta)$ is visual distance exactly $\epsilon$ from $p$.  Since the arc $a_t$ does not meet $e^+(a_0,\theta)$ we have $g_ia_t\in B_{\epsilon}(q)$ for all $t$ and sufficiently large $i$.    See Figure \ref{fig:onlytheta} for a schematic illustration.  
Consider the sets $g_i S(a_{t_i},\theta) = S(g_i(a_{t_i}), \rho(g_i)(\theta))$.
Recall by Equation \eqref{eq:equivariance} we have
\[e^+(g_i a_t, \rho(g_i) \theta) = g_i \cdot e^+(a_t, \theta) . \]
This implies that the geodesics shadowed by the sets $S(g_i(a_{t_i}), \rho(g_i)(\theta))$ all pass through some compact subset of the Cayley graph $\Gamma$, and so the sets themselves all meet some compact $K\subset X$.

For each $i$, fix a point $y_i \in L_{g_i a_{t_i}} \cap Y_{\rho(g_i)(\theta)} \cap K$.
After passing to a further subsequence, the points $y_i$ converge to some $y_\infty \in K$.   Now $g_i a_{t_i} \to q$ and $\rho(g_i)(\theta) \to \theta_\infty$, so $y_\infty \in L_q \cap Y_{\theta_\infty} \cap K$.    
Using the notation from Proposition~\ref{prop:cont_X}, continuity of positive endpoints implies that $\mathcal{E}^+(y_\infty) = \lim_{n \to \infty} \mathcal{E}^+(y_i)$ which is by construction some point at distance $\epsilon$ from $p$.  

Now consider instead the constant sequence $t = 0$ instead of $t_i$.  By the same reasoning, for $i$ sufficiently large, $L_{g_i a_{0}} \cap Y_{\rho(g_i)(\theta)}$ will contain a point $z_i$ in $K$.  After passing to a subsequence, these converge to a point $z_\infty \in L_q \cap Y_{\theta_\infty} \cap K$.   By continuity of $\mathcal{E}^+$ we have 
$\mathcal{E}^+(z_\infty) = \lim_{n \to \infty} \mathcal{E}^+(z_i) = p$.   Thus, we have found two points, $y_\infty$ and $z_\infty$, both in $Y_{\theta_\infty} \cap L_q$ with different positive endpoints $\mathcal{E}^+(z_\infty) \ne \mathcal{E}^+(y_\infty)$.
This directly contradicts Proposition~\ref{prop:gromov_prod}, and this contradiction concludes the proof. 
\end{proof}

Recall that item~\eqref{itm:slices} from Lemma~\ref{lem:D01} says that, if $(a, \theta, c)$ and $(a', \theta, c')$ lie in $D_1$, then there exists a path $a_t$ with $a_0 = a$ and $a_1 = a'$ and a point $c''$ such that $(a_t, \theta, c'') \in D_1$ for all $t$.  Proposition~\ref{prop:cont_endpt} says that the map $e^+(a_t, \theta)$ is therefore continuous at each point, and thus by Lemma~\ref{lem:only_theta} we conclude it is constant.   In summary, we have the following.  
\begin{corollary} \label{cor:only_theta}
If $(a, \theta, c)$ and $(a', \theta, c')$ lie in $D_1$, then $e^+(a, \theta) = e^+(a', \theta)$.  
\end{corollary} 

\begin{definition}
Define $h:S^{n-1} \to S^{n-1}$ by $h(\theta) = e^+(a, \theta)$ where $a$ is any point such that there exists $c$ with $(a, \theta, c) \in D_0$.   
\end{definition} 

Note that $h$ is defined everywhere, and is continuous by the continuity of $e^+(a, \theta)$ given by Proposition~\ref{prop:cont_endpt}.   It remains to check that $h$ satisfies the other properties of the semi-conjugacy required to prove Theorem \ref{thm:main}.  The second point in Proposition~\ref{prop:gromov_prod} states that $\gprod{h(\theta)}{\theta}{\idG} \ge R-(4H + 11\delta)$.   Our choice of $R$ in Lemma \ref{lem:D01} \eqref{eq:R} implies that $\gprod{h(\theta)}{\theta}{\idG} \ge C_V$, satisfying equation \eqref{eq:smallsemiconjugacy} as desired, and showing that $h$ lies in our chosen neighborhood $V$.  This neighborhood contains only degree one maps, so $h$ is surjective.

We now check equivariance.  Let $g$ be an element of the generating set $\mathcal{S} \cup \mathcal{S}^{-1}$ used in the definition of $\Gamma$ and let $\theta$ be given.  
Choose $a$ so that $(a, \theta, c) \in D_0$.  We have
\[\rho_0(g) h(\theta) = g e^+(a, \theta) = e^+(g a, \rho(g) \theta) \] 

By definition of $D_{\frac{1}{2}}$, and our conditions on $\rho$, we have $(ga, g\theta, gc) \in D_{\frac{1}{2}}$.  By definition of $D_1$, and item \eqref{item:smallgens} of Definition~\ref{def:U(V)} which defines the neighborhood 
$\mathcal U(V)$, we then have $(ga, \rho(g)\theta, gc) \in D_1$.   Thus, by Corollary~\ref{cor:only_theta}, 
$e^+(ga, \rho(g) \theta, gc) = e^+(a', \rho(g)\theta, c')$ for any choice of $a'$ and $c'$ such that $(a', \rho(g)\theta, c') \in D_0$, thus giving 
\begin{equation}\label{generators}
\rho_0(g) h(\theta) = h(\rho(g) \theta)
\end{equation}
for all $\theta\in \bgam$.
Since \eqref{generators} holds for generators of $G$, it holds, inductively, for all elements $g \in G$.   This shows that $h$ is a semiconjugacy in the specified neighborhood of the identity map of $S^n$, completing the proof of the theorem. \qed

\end{document}